\documentclass[12pt]{amsart}
\usepackage{amscd}
\usepackage{amssymb}
\usepackage{graphicx}
\usepackage{color}
\usepackage[centering,text={15.5cm,22cm},
		marginparwidth=20mm]{geometry}
\usepackage{hyperref}

\usepackage{mathrsfs}

\usepackage[all]{xy}

\newtheorem{theorem}{Theorem}[section]

\newtheorem{lemma}[theorem]{Lemma}

\newtheorem{proposition}[theorem]{Proposition}
\newtheorem{corollary}[theorem]{Corollary}
\newtheorem{conjecture}[theorem]{Conjecture}

\theoremstyle{definition}
\newtheorem{definition}[theorem]{Definition}
\newtheorem{definition-lemma}[theorem]{Definition-Lemma}

\newtheorem{construction}[theorem]{Construction}

\newtheorem{example}[theorem]{Example}

\theoremstyle{remark}
\newtheorem{remark}[theorem]{Remark}

\numberwithin{equation}{section}
\numberwithin{figure}{section}

\newcommand{\ZZ} {\mathbb{Z}}
\newcommand{\QQ} {\mathbb{Q}}
\newcommand{\RR} {\mathbb{R}}
\newcommand{\bR} {\RR}

\newcommand{\PP} {\mathbb{P}}

\renewcommand{\AA} {\mathbb{A}}
\newcommand{\GG} {\mathbb{G}}

\newcommand {\s}  {{\bf s}}
\newcommand {\bfs} {\s} 
\newcommand {\shA}  {\mathcal{A}}

\newcommand {\shD}  {\mathcal{D}}

\newcommand {\shI}  {\mathcal{I}}

\newcommand {\shL}  {\mathcal{L}}

\newcommand {\shO}  {\mathcal{O}}

\newcommand {\shR}  {\mathcal{R}}
\newcommand {\shS}  {\mathcal{S}}

\newcommand {\shX}  {\mathcal{X}}

\newcommand {\ft} {{\mathrm{ft}}}

\newcommand {\foT}  {\mathfrak{T}}
\newcommand {\foU}  {\mathfrak{U}}

\newcommand {\fom}  {\mathfrak{m}}

\newcommand {\Cl}  {\operatorname{Cl}}

\newcommand {\coker} {\operatorname{coker}}

\newcommand {\Gm} {\GG_m}

\newcommand {\MW} {\operatorname{MW}}

\newcommand {\prin}  {\mathrm{prin}}

\newcommand {\Hom}  {\operatorname{Hom}}

\newcommand {\ind}  {\operatorname{ind}}

\newcommand {\kk} {\Bbbk}

\renewcommand {\max} {{\operatorname{max}}}

\newcommand {\Pic}  {\operatorname{Pic}}
\newcommand {\PGL}  {\operatorname{PGL}}

\newcommand {\rank} {\operatorname{rank}}

\newcommand {\sat}  {\operatorname{sat}}

\newcommand {\Sing} {\operatorname{Sing}}

\newcommand {\Spec} {\operatorname{Spec}}

\newcommand {\TV}{\operatorname{TV}}
\newcommand {\tTV}{\widetilde{\TV}}

\newcommand {\Cox}{\operatorname{Cox}}

\newcommand {\sgn} {\operatorname{sgn}}

\newcommand {\trop} {\mathrm{trop}}

\newcommand {\NE}   {\operatorname{NE}}

\newcommand {\up}  {\mathrm{up}}
\newcommand {\Mat} {\operatorname{Mat}}

\newcommand {\uf} {\mathrm{uf}}

\def\mapright#1{\smash{
  \mathop{\longrightarrow}\limits^{#1}}}

\def\oM{\overline{M}}

\def\bP{\Bbb P}

\def\cO{\Cal O}

\def\rank{\operatorname{rank}}

\def\bZ{\Bbb Z}

\def\bQ{\Bbb Q}
\def\bG{\Bbb G}
\def\bR{\Bbb R}
\def\bA{\Bbb A}

\def\oY{\bar{Y}}

\def\oD{\bar{D}}

\def\oSigma{\bar{\Sigma}}

\def\cA{\Cal A}
\def\cD{\Cal D}

\def\tS{\tilde S}

\def\cA{\Cal A}

\def\oS{\overline{S}}

\def\tS{\tilde{S}}

\def\oN{\overline{N}}

\def\cY{\Cal Y}

\def\cD{\Cal D}
\def\cX{\Cal X}

\def\oV{\overline{V}}

\def\Spec{\operatorname{Spec}}

\def\PGL{\operatorname{PGL}}

\def\cO{\Cal O}

\def\Sing{\operatorname{Sing}}

\def\Pic{\operatorname{Pic}}
\def\Spec{\operatorname{Spec}}

\def\Hom{\operatorname{Hom}}
\def\Cal{\mathcal}

\def\Efi#1#2#3#4#5{\displaystyle
#1\!\!-\!\!#2
\!\!-\!\!#3
\!\!-\!\!#4
\hskip-24.2pt\lower4.5pt\hbox{${\scriptstyle|}
\hskip-3.35pt\lower6pt\hbox{$#5$}$}}

\def\Evia#1#2#3#4#5{\displaystyle
#1\!\!-\!\!#2
\!\!-\!\!#3
\hskip-24.2pt\lower4.5pt\hbox{${\scriptstyle|}
\hskip-3.35pt\lower6pt\hbox{$#4\!\!-\!\!\!-\!\!\!-\!\!$}$\hskip2.3pt${\scriptstyle|}
\hskip-3.35pt\lower6pt\hbox{$#5$}$}}

\def\Ezia#1#2#3#4{\displaystyle
#1\!\!-\!\!#2
\hskip-14.8pt\lower4.5pt\hbox{${\scriptstyle|}
\hskip-3.35pt\lower6pt\hbox{$#3\!\!-\!\!$}$\hskip2.3pt${\scriptstyle|}
\hskip-3.35pt\lower6pt\hbox{$#4$}$}}

\def\Efia#1#2#3#4#5#6{\displaystyle
#1\!\!-\!\!#2
\!\!-\!\!#3
\!\!-\!\!#4
\hskip-24.2pt\lower4.5pt\hbox{${\scriptstyle|}
\hskip-3.35pt\lower6pt\hbox{$#5$}$\hskip5.7pt${\scriptstyle|}
\hskip-3.35pt\lower6pt\hbox{$#6$}$}}

\def\Esi#1#2#3#4#5#6{\displaystyle
#1\!\!-\!\!#2
\!\!-\!\!#3
\!\!-\!\!#4\!\!-\!\!#5
\hskip-24.2pt\lower4.5pt\hbox{${\scriptstyle|}
\hskip-3.35pt\lower6pt\hbox{$#6$
\lower3pt\hbox{\ }}$}}

\def\Esia#1#2#3#4#5#6#7{\displaystyle

#1\!\!-\!\!#2
\!\!-\!\!#3
\!\!-\!\!#4\!\!-\!\!#5
\hskip-24.2pt\lower4.5pt\hbox{${\scriptstyle|}
\hskip-3.35pt\lower6pt\hbox{$#6$\hskip-3.8pt\lower4.5pt\hbox{${\scriptstyle|}
\hskip-3.35pt\lower6pt\hbox{$#7$}$}}
\lower3pt\hbox{\ }$}}

\def\Ese#1#2#3#4#5#6#7{\displaystyle
#1\!\!-\!\!#2
\!\!-\!\!#3
\!\!-\!\!#4\!\!-\!\!#5\!\!-\!\!#6
\hskip-33.6pt\lower4.5pt\hbox{${\scriptstyle|}
\hskip-3.35pt\lower6pt\hbox{$#7$
\lower3pt\hbox{\ }
}$}}

\def\Esea#1#2#3#4#5#6#7#8{\displaystyle
#1\!\!-\!\!#2
\!\!-\!\!#3
\!\!-\!\!#4\!\!-\!\!#5\!\!-\!\!#6\!\!-\!\!#7
\hskip-33.6pt\lower4.5pt\hbox{${\scriptstyle|}
\hskip-3.35pt\lower6pt\hbox{$#8$
\lower3pt\hbox{\ }
}$}}

\def\Eei#1#2#3#4#5#6#7#8{\displaystyle
#1\!\!-\!\!#2
\!\!-\!\!#3
\!\!-\!\!#4\!\!-\!\!#5\!\!-\!\!#6\!\!-\!\!#7
\hskip-43.2pt\lower4.5pt\hbox{${\scriptstyle|}
\hskip-3.35pt\lower6pt\hbox{$#8$
\lower3pt\hbox{\ }
}$}}

\def\Eeia#1#2#3#4#5#6#7#8#9{{\displaystyle
#1\!\!-\!\!#2
\!\!-\!\!#3
\!\!-\!\!#4\!\!-\!\!#5\!\!-\!\!#6\!\!-\!\!#7\!\!-\!\!#8
\hskip-52.2pt\lower4.5pt\hbox{${\scriptstyle|}
\hskip-3.35pt\lower6pt\hbox{$#9$
\lower3pt\hbox{\ }
}$}}}

\def\trop{\operatorname{trop}}
\def\oM{\overline{M}}
\def\os{\overline{S}}

\def\tS{\tilde{S}}
\def\ts{\tS}
\def\ts7{\tilde{S}_7}

\def\bP{\Bbb P}

\def\Hom{\operatorname{Hom}}

\def\cO{\Cal O}

\def\cD{\Cal D}

\def\Pic{\operatorname{Pic}}

\def\rank{\operatorname{rank}}

\def\bZ{\Bbb Z}

\def\bQ{\Bbb Q}
\def\bG{\Bbb G}
\def\bR{\Bbb R}
\def\bA{\Bbb A}

\def\cA{\Cal A}
\def\cD{\Cal D}

\def\tS{\tilde S}

\def\cA{\Cal A}

\def\oS{\overline{S}}

\def\tS{\tilde{S}}

\def\oN{\overline{N}}

\def\cY{\Cal Y}
\def\cX{\Cal X}

\def\Spec{\operatorname{Spec}}

\def\PGL{\operatorname{PGL}}

\def\Ker{\operatorname{Ker}}

\def\Sing{\operatorname{Sing}}
\def\Cal{\mathcal}
\def\Pic{\operatorname{Pic}}

\def\oN{\overline{N}}
\def\os7p{\oS_7'}
\def\os7{\oS_7}
\def\on6{\oN_6}
\def\n6{\oN_6}

\def\brg0{(\bR_{\geq 0})}

\def\mydate{\ifcase\month \or January\or February\or March\or
April\or May\or June\or July\or August\or September\or October\or 
November\or December\fi \space\number\day,\space\number\year}

\begin{document}

\title[Birational geometry of cluster algebras]
{Birational geometry of cluster algebras}
\author{Mark Gross} 
\address{DPMMS, Centre for Mathematical Sciences, Wilberforce Road,
Cambridge, CB3 0WB United Kingdom}
\email{mgross@dpmms.cam.ac.uk}
\author{Paul Hacking}
\address{Deptartment of Mathematics and Statistics, Lederle Graduate
Research Tower, University of Massachusetts, Amherst, MA 01003-9305}
\email{hacking@math.umass.edu}

\author{Sean Keel}
\address{Department of Mathematics, 1 University Station C1200, Austin,
TX 78712-0257}
\email{keel@math.utexas.edu}
\begin{abstract} 
We give a geometric interpretation of cluster varieties in terms of 
blowups of toric varieties. This enables us to provide, among other
results, an elementary geometric proof of the Laurent phenomenon for
cluster algebras (of geometric type), extend
Speyer's example \cite{Sp13} of 
upper cluster algebras which are not finitely generated,
and show that the Fock-Goncharov dual basis conjecture is usually false.
\end{abstract}

\maketitle
\tableofcontents
\bigskip

\section*{Introduction}

Cluster algebras were introduced by Fomin and Zelevinsky in \cite{FZ02a}.
Fock and Goncharov introduced a more geometric point of view in 
\cite{FG09}, introducing the $\cA$ and $\cX$ cluster varieties constructed
by gluing together ``seed tori'' via birational maps known as cluster
transformations. 

In this note, motivated by our study of log Calabi-Yau varieties initiated
in the two-dimensional case in \cite{GHK11}, 
we give a simple alternate explanation of basic constructions in the theory
of cluster algebras in terms of blowups of toric varieties. Each seed
roughly
gives a description of the $\cA$ or $\cX$ cluster variety\footnote{More 
precisely, the $\shA_{\prin}$, $\shA_t$ (defined in
\S \ref{clusterreviewsection}) or $\shX$ cluster variety.} 
as a blowup of a toric variety, and a mutation of the seed corresponds
to changing the blowup description by an elementary transformation of a 
$\bP^1$-bundle.
Certain global features of the cluster variety not obvious from the expression as a union
of tori are easily seen from this construction. For example, it gives a simple
geometric explanation for the Laurent phenomenon (originally proved in \cite{FZ02b}), 
see Corollary \ref{lpcor}. From the blowup picture  it is clear that the Fock-Goncharov dual
basis conjecture, particularly the statement that tropical points of the
Langlands dual $\cA$ parameterize a natural basis of regular functions on
$\cX$, can fail frequently, see \S \ref{FGcounterexamplessection}. 

In more detail, in \S\ref{motivationsection}, we explain the basic
philosophical point of view demonstrating how a study of log Calabi-Yau
varieties can naturally lead to the basic notions of cluster algebras.
This section can be read as an extended introduction; its role in the paper is purely
motivational. In \S\ref{clusterreviewsection}, 
we review the definitions of cluster varieties, following \cite{FG09}. We pay
special attention to the precise procedure for gluing tori via cluster
transformations, as this has not been discussed to the precision we need
in the literature.

\S\ref{gcvsec} is the heart of the paper. Here we describe how cluster
transformations, which a priori are birational maps between algebraic tori,
can be viewed naturally as isomorphisms between blowups of certain associated
toric varieties. In this manner, cluster transformations can be 
interpreted as elementary transformations, a standard procedure for
modifying $\PP^1$-bundles in algebraic geometry. This procedure blows up
a codimension two center in a $\PP^1$-bundle meeting any $\PP^1$ fibre
in at most one point, and blows down the proper
transform of the union of $\PP^1$ fibres meeting the center.

This is a very general construction, covered in \S\ref{elementarysection};
we then specialise to the case of the $\cA$ and $\cX$ cluster varieties
in \S\ref{codimtwosection}. Unfortunately, our construction
does not work in general
for the $\cA$ cluster variety, but does work for the $\cA$ variety
\emph{with principal coefficients}. This variety $\shA_{\prin}$ fibres
over an algebraic torus with $\cA$ being the fibre over the identity element
of the torus. Properties such as the Laurent phenomenon for $\cA$ can be
deduced from that for $\cA_{\prin}$. Many of the phenomena discussed here
also work for a very general fibre $\shA_t$ 
of the map from $\shA_{\prin}$; we call
such a cluster variety an $\shA$ cluster variety \emph{with general 
coefficients}. The algebra of regular functions of such a cluster variety
are of the kind considered by Speyer in \cite{Sp13}.

The key result is Theorem 
\ref{maingeometrictheorem}, which gives the precise description of 
the $\cX$, principal $\cA$ cluster varieties and $\cA$ cluster varieties
with general coefficients up to codimension two in terms
of a blowup of a toric variety. The toric variety and the center of
the blowup is specified very directly by the seed data determining
the cluster variety. An immediate consequence is the Laurent phenomenon,
Corollary \ref{lpcor}.

In \S\ref{univtorsorsection}, we give another description 
of the principal $\cA$ cluster variety and $\cA$ cluster variety with
general coefficients in terms of line bundles on the $\cX$ cluster variety.
There is in fact an algebraic torus which acts on $\cA_{\prin}$, and the
quotient of this action is $\cX$, making $\cA_{\prin}$ a torsor of $\cX$.
We give a precise description of this family in terms of line bundles on
$\cX$. Furthermore, there are tori $T_{K^*}$ and $T_{K^{\circ}}$ such that
there is a map $\cX\rightarrow T_{K^*}$ and an action of $T_{K^{\circ}}$ on
any $\cA$ cluster variety with general coefficients determined by the
seed data. We show that for any such sufficiently general $\cA$ cluster
variety $\cA_t$, there is a $\phi=\phi(t)\in T_{K^*}$ such that up
to codimension two, $\cA_t$ is the \emph{universal torsor} of $\cX_{\phi}$, 
essentially obtained as ${\bf Spec}\,
\bigoplus_{\shL\in \Pic(\shX_{\phi})}\shL$. In particular, this allows us
to identify the corresponding upper cluster algebra
with the Cox ring of $\shX_\phi$. This is a slight simplification of the
discussion: see the main text for precise statements. 
The Cox ring of any variety
with finitely generated torsion free Picard group is factorial, see \cite{Ar08} and 
\cite{BH03}. This explains
the ubiquity of factorial cluster algebras remarked on, e.g., in \cite{K12},
\S 4.6. 

The remainder of the paper now restricts to the case that the skew-symmetric
matrix determining the cluster algebra has rank $2$. This case is
quite easy to interpret geometrically, since now the family $\cX\rightarrow
T_{K^*}$ is a family of surfaces. In fact, the fibres are essentially
the interiors of \emph{Looijenga pairs}.
A Looijenga pair is a pair $(Y,D)$ where $Y$ is a rational surface
and $D\in |-K_Y|$ is a cycle of rational curves, $U := Y \setminus D$ is 
the interior.  We study moduli of such pairs 
in \cite{GHK12}. Here, we show (Theorem \ref{Xisuniversalfamily}) 
that essentially $\cX\rightarrow T_{K^*}$ 
coincides with a type of universal family constructed in \cite{GHK12}.
Our construction implies that in many cases, the kernel of
the skew-symmetric matrix carries a canonical symmetric form, invariant
under mutations, see Theorem \ref{symmetricform}. Though not (as far
as we know) previously observed, this symmetric form controls the gross 
geometry of $\cX$, in particular the generic fibre of 
$\cX \rightarrow T_{K^*}$. Indeed,
the fibre is affine if and only if 
the form is negative definite; when the form is indefinite
the fibres are the complement of a single point
in a compact complex analytic space, and thus have no non-constant 
global functions. Thus
in this indefinite case (which from the blowup point of view is the generic
situation) the only global functions on $\cX$ are pulled back from $T_{K^*}$,
contradicting the dual basis conjecture of \cite{FG09}, see
\S \ref{FGcounterexamplessection}.

In \S\ref{nfgsec}, we give a general procedure for constructing
upper cluster algebras with general or principal coefficients which are
not finitely generated. These examples generalize that given by Speyer
in \cite{Sp13}, and suggest that ``most'' upper cluster algebras 
are not finitely generated. These examples arise because Cox rings tend
not to be finitely generated. Indeed, 
finite generation of the Cox ring of a projective
variety is a very strong (Mori Dream Space) condition, see \cite{HK00}. 
\medskip

In this paper, we will always work over a field $\kk$ of characteristic
zero.

\medskip

\emph{Acknowledgments}: The genesis of
our results 
on cluster varieties 
was a conversation with M. Kontsevich. He pointed out to us that 
(in the skew-symmetric case) a seed is the same thing as a collection of vectors in a symplectic lattice 
and the piecewise linear cluster mutation is just like {\it moving worms} in
the integral affine manifolds central to mirror symmetry for open Calabi-Yau
varieties,  see e.g., \cite{GHK11}. 
Before this conversation we had been incorrectly 
assuming the cluster picture was a very special case of the mirror construction in 
\cite{GHK11}. However,
Kontsevich's remarks led us to the correct view that in dimension two,
the scope of the two theories are exactly the same. This in turn led
to the simple blowup description of cluster varieties we describe here. 

We first learned of the connection between mirror symmetry
and cluster varieties from conversations with A. Neitzke. We received considerable
inspiration from conversations with V.\ Fock, S.\ Fomin, A.\ Goncharov, 
B.\ Keller, B.\ Leclerc, G.\ Musiker, M.\ Shapiro, Y.\ Soibelman, and
D.\ Speyer. Special thanks go to Greg
Muller, who pointed out a crucial mistake in a draft version of this paper, 
see Remark \ref{ourscrewup}.

The first author
was partially supported by NSF grant DMS-1105871 and DMS-1262531, the second by
NSF grants DMS-0968824 and DMS-1201439, and the third by NSF grant DMS-0854747.

\section{Log Calabi-Yau varieties and a geometric motivation for
cluster varieties.}
\label{motivationsection}

To a geometer, at least to the three of us, the definition of a cluster
algebra is rather bizarre and overwhelming. Here we explain the geometric
motivation in terms of log Calabi-Yau varieties.
There are two
elementary constructions of log Calabi-Yau (CY) varieties. 
The first method is to glue together tori in such a way that the volume forms 
patch. The second method
is to blow up a toric variety along a codimension two center which 
is a smooth
divisor in a boundary divisor, and then remove
the strict transform of the toric boundary. As we will see, 
the simplest instances of either construction are closely related, and
either leads to cluster varieties. The first approach 
extends the viewpoint of \cite{FG09}, the second was inspired by 
\cite{L81}. 

\begin{definition} \label{cydef} Let $(Y,D)$ be a smooth projective variety with a 
normal crossing divisor, and let $U= Y \setminus D$. By \cite{I77}, 
the vector subspace
\[
H^0(Y,\omega_Y(D)^{\otimes m}) \subset H^0(U,\omega_U^{\otimes m})
\]
(where the inclusion is induced by restriction) 
depends only on $U$, i.e., is independent of the choice of normal crossing
compactification. We say $U$ is \emph{log Calabi-Yau} if for all $m$ this 
subspace is one-dimensional, generated by 
$\Omega^{\otimes m}$ for a volume (i.e., nowhere
vanishing) form $\Omega \in H^0(U,\omega_U)$. Note that by definition
$\Omega$ is unique up to scaling. 
\end{definition}


In practice, log Calabi-Yau varieties
are often recognized using the following:

\begin{lemma} Let $(Y,D)$ be a dlt pair with $K_Y + D$ 
trivial (in particular Cartier), and $Y$ projective. Let $U \subset Y \setminus D$
be a smooth open subset, with $(Y \setminus D ) \setminus U$ of codimension 
at least $2$. Then $U$ is log CY. \end{lemma} 

For the definition of dlt (divisorial log terminal), see \cite{KM98}, 
Def.\ 2.37. As this section should be viewed as purely motivational, the
reader who wishes to avoid the technicalities of the minimal model
program should feel free to assume that the pair $(Y,D)$ is in fact normal
crossings.

\begin{proof} 
When $(Y,D)$ has normal crossings this is immediate from Definition \ref{cydef}. The definition of dlt
is such that the vector space of Definition \ref{cydef} can be computed using 
a dlt (instead of normal crossing) compactification. 
\end{proof}

\begin{remark}
The data $(Y,D)$ with $U\subseteq Y\setminus D$ as in the lemma
is called a \emph{minimal model} for $U$.
One example of a minimal model is a pair $(Y,D)$ with $D\in |-K_Y|$ a
reduced normal crossings divisor. This is a minimal model for $U=Y\setminus
D$. The main conjectures of the minimal model program would 
imply every log CY has a minimal model, see \cite{BCHM}.
\end{remark} 

\begin{lemma} \begin{enumerate}
\item
Let $U \subset V$ be an open subset, with $(U,\Omega)$ log CY. Then $V$ is
log CY if and only if $\Omega$ extends to a volume form on $V$, and in this case
$\Omega$ is a scalar multiple of the volume form of $V$.
\item Let $\mu: U \dasharrow V$ be a birational map between smooth
varieties which is an isomorphism outside codimension two subsets of
the domain and range. Then $U$ is log CY if and only if $V$ is.
\end{enumerate}
\end{lemma}

\begin{proof} 
For (1), if $V$ is log CY, then clearly its volume form restricts 
to a scalar multiple of the volume
form on $U$. Now suppose $U$ is log CY, and its volume form $\Omega$ extends to
a volume form on $V$. We have $U\subseteq V\subseteq Y$ where $Y$
is a compactification of both $U$ and $V$. Thus $\Omega$ (and its powers) 
obviously has at worst simple poles on any divisor contained in
$Y\setminus V$, and it is unique in this respect, since we have
the same properties for $\Omega$ as a volume form on $U$.
Next (2) follows from (1), passing to the open
subsets where the map is an isomorphism, noting that in (1), when the
complement of $U$ has codimension at least two, the extension condition is 
automatic.
\end{proof} 

\begin{definition} We say a log CY $U$ has \emph{maximal boundary} if it has
a minimal model $(Y,D)$ with a zero-dimensional log canonical center.
For example, this is the case
if $(Y,D)$ is a minimal model for $U$ such that 
$D$ is simple normal crossings and contains a zero-dimensional stratum, i.e.,
a point which is the intersection of 
$\dim(Y)$ distinct irreducible components
of $D$.
\end{definition}

\begin{example} \label{dbex} Consider the group $G = \PGL_{n}$. 
There are the $2n-1$ minors of an $n\times n$ matrix given by the square
submatrices in the upper right corner or the lower left corner.
For example, for $n=3$ these are the $4$ minors
 \[
a_{1,3}, a_{3,1}, 
\left| \begin{matrix}
a_{1,2} & a_{1,3} \\
a_{2,2} & a_{2,3} \\
\end{matrix} \right|, 
\left| 
\begin{matrix}
a_{2,1} & a_{2,2} \\
a_{3,1} & a_{3,2} \\
\end{matrix} \right |
\]
and the determinant of the $3\times 3$ matrix itself. 
Let $D \subset Y = \bP(\Mat_{n \times n})= \bP^{n^2-1}$
be the union of the $2n-1$ divisors given by the zero locus of these minors. 
Note the total degree of $D$ is
\[
1 + 2 + \dots + (n-1) + 1 + 2 + \dots + (n-1) + n = n^2,
\]
so $D \in |-K_Y|$. With some non-trivial effort, 
one can check $(Y,D)$ is dlt, with a 
zero-dimensional log canonical center,  and thus
$(Y,D)$ is a minimal model for the smooth affine log CY with maximal boundary
$U \subset G$, the non-vanishing locus of this collection of minors. 
$U$ is by definition the open double Bruhat cell in $G$.
\end{example}

A log CY $U$ 
with maximal boundary will (in dimension at least two) always have infinitely many minimal models.
The set of possibilities 
leads to a fundamental invariant: 

\begin{definition} Let $(U,\Omega)$ be a log CY. Define
\begin{equation} \label{uteq}
\begin{aligned}
U^{\trop}(\bZ) &:= \{\text{divisorial discrete valuations } 
v: k(U) \setminus \{0\} \to \bZ\,|\, v(\Omega) < 0 \}  \cup \{0\} \\
             &:=  \{(E,m)\,|\, m \in \bZ_+, E \subset 
           (Y \setminus U),\, \Omega \text{ has a pole along } E \} \cup \{0\}.
\end{aligned}
\end{equation}
Here $k(U)$ is the field of rational functions of $U$, 
a discrete valuation is called divisorial if
it is given by the order of vanishing of a divisor 
on some variety birational to $U$. Furthermore,
we define 
\[
v(g dz_1\wedge\cdots\wedge dz_n):=v(g)
\]
for $z_1,\ldots,z_n$ local
coordinates in a neighborhood of the generic point of the divisor corresponding
to $v$; this is independent of the choice of coordinates as a change
of coordinates only changes $g$ by a unit.
In the second expression $E$ is a divisorial irreducible component of the boundary in 
some partial compactification $U \subset Y$, and 
two divisors on two possibly different birational varieties are
identified if they give the same valuation on their common field of fractions.
\end{definition}

The simplest example of a log CY with maximal boundary is an algebraic torus 
\[
T_N:= N\otimes \bG_m,
\]
for $N = \bZ^n$. Note $H^0(T_N,\cO_{T_N}) = \kk[M]$, where $M=\Hom(N,\ZZ)$
is the character lattice of $T_N$.

\begin{lemma} \label{evallem} Restriction of valuations to the character lattice $M$ induces a 
canonical isomorphism 
\[
T_N^{\trop}(\bZ) = N.
\]
A minimal model for $T_N$ is the same as a complete $T_N$-equivariant 
toric compactification. 
\end{lemma}

\begin{proof} This is an easy log discrepancy computation, using
e.g., \cite{KM98}, Lemmas 2.29 and 2.45.
\end{proof}

Thus $U^{\trop}(\bZ)$ gives an analog for any log CY of the cocharacter 
lattice of a torus. Note however that
in general $U^{\trop}(\bZ)$ is not a group as addition does not make sense.
We conjecture there is also an analog of the character lattice, or equivalently,
the dual torus:
\begin{conjecture}\cite{GHK11}
\label{conjGHKI} Let $(Y,D)$ be a simple normal crossings minimal model for
a log CY with maximal boundary $U = Y \setminus D$, and assume $D$ supports an ample
divisor (note this implies $U$ is affine). Let $R = \kk[\Pic(Y)^*]$. 
The free $R$-module 
$V$ with basis $U^{\trop}(\bZ)$ 
has a natural finitely generated $R$-algebra structure whose
structure constants are non-negative integers 
determined by counts of rational curves on $U$. The associated fibration 
$p:\Spec(V) \to \Spec(R) = T_{\Pic(Y)}$ is a flat family of affine log CYs with maximal boundary. 
Letting $K$ be the kernel of the natural surjection
$\Pic(Y) \twoheadrightarrow \Pic(U)$, $p$ is $T_K$-equivariant. The quotient family
$\Spec(V)/T_K \to T_{\Pic(U)}$ depends only on $U$ (is independent of the choice of minimal model), and is
the mirror family to $U$. 
\end{conjecture}

\begin{remark}  An analog of Conjecture \ref{conjGHKI} is
expected for compact Calabi-Yaus, but perhaps only with formal (e.g., Novikov) 
parameters,
and for Calabi-Yaus near the so-called large complex
structure limit. This will be discussed in forthcoming work.
The maximal boundary condition means the boundary is highly degenerate 
--- we are thus already in some sense
in the large complex structure limit,  and so one can hope that 
no formal power series or further limits are required. 
This is one reason to focus on this case. The other is the wealth of fundamental
examples. 

The conjecture is of interest independently of mirror symmetry: in many 
instances the
variety $U$ and its prospective mirror are known varieties of compelling 
interest. The
conjecture then gives a new construction of a variety we already care about, 
a construction which in particular
endows the mirror (and each fibre of the family) with a canonical basis of functions. In any case 
mirror symmetry is conjecturally an involution, the mirror
of the mirror being a family of deformations of the original $U$. Thus the conjecture says in particular
that any affine log CY with maximal boundary is a fibre of the output of such a 
construction, and thus in particular has a canonical basis of functions, 
$B_U$.
One then expects $B_U$ to be the tropical set of the conjectural mirror. 

We call a partial compactification $U \subset Y$ a \emph{partial
minimal model} if the volume form $\Omega$ has a pole on every irreducible divisorial
component of $Y \setminus U$. One checks using Lemma \ref{evallem} that a partial minimal
model for an algebraic torus is the same thing as a toric variety.
We further
conjecture that for any partial minimal model (not necessarily affine) of 
an affine log CY $U$ with maximal boundary, 
$B_U \cap H^0(Y,\cO_Y) \subset H^0(Y,\cO_Y)$ is a basis of regular functions on $Y$. For example
we conjecture that the open double Bruhat cell $U \subset G$ has a canonical basis of functions,
and that the subset of basis elements which extend regularly to $G$ give a basis of functions on
$G$. 
\end{remark} 

After tori, the next simplest example of a log CY with maximal boundary is 
obtained by  gluing together algebraic tori in such a way that the volume forms patch.
More precisely, suppose that
\[
\cA = \bigcup_{\bfs \in \shS} T_{N,\bfs}
\]
is a variety covered by open 
copies of the torus $T_N$ indexed by the set $\shS$. 
This gives canonical 
birational maps $\mu_{\bfs,\bfs'}:T_{N,\bfs} \dasharrow T_{N,\bfs'}$ for each
pair of {\it seeds} $\bfs,\bfs' \in \shS$. Then $\cA$ will be log CY if and
only if
each birational map is a \emph{mutation}, i.e., preserves the volume
form: $\mu^*(\Omega) = \Omega$. In this case each choice of seed torus
$T_{N,\bfs} \subset \cA$ 
gives a canonical identification $\cA^{\trop}(\bZ) = 
T^{\trop}_{N,\bfs}(\bZ) = N$. 

We can reverse the procedure. Beginning with a collection of such mutations 
satisfying the cocycle condition,  we can 
canonically glue together the tori along the maximal open sets where the maps
are isomorphisms to form a log CY $\cA$. See Proposition 
\ref{gluingconstruction} for details.  
The simplest example
of a mutation comes from a pair $(n,m) \in N \times M$ with $\langle n,m
\rangle :=m(n) = 0$. It is defined by 
\begin{equation} \label{of}
\mu_{(n,m)}^*(z^{m'}) = z^{m'} \cdot (1 + z^m)^{\langle m',n\rangle}
\end{equation} 
where $z^{m'}, z^m \in \kk[M]$ are the corresponding characters of $T_N$. 
Cluster varieties
are log CYs formed by gluing tori by mutations of this simple sort 
(and compositions of such) for a particular parameterizing set $\shS$. 
See \S \ref{clusterreviewsection} for details. 

Though these are the simplest non-toric log CYs, 
there are already very interesting
examples, including double Bruhat cells for reductive groups, their flag varieties and unipotent radicals,
and character varieties of punctured Riemann surfaces. See \cite{BFZ05} and 
\cite{FG06}.

Note that these simple
mutations come in obvious dual pairs --- we can simply reverse the order and
consider 
\begin{equation} \label{dequation}
(m,-n) \in \Hom(N \times M,\ZZ)=M\times N,
\end{equation}
so that $\mu_{(m,-n)}$ defines a birational automorphism of $T_M$.
Thus for each $\cA := \bigcup_{\bfs \in \shS} T_{N,\bfs}$
built from such maps, there is a canonical {\it dual} $\cX := 
\bigcup_{\bfs \in \shS} T_{M,\bfs}$,
just obtained by replacing each torus (and each mutation) by its dual.  
For the particular parameterizing 
set $\shS$ used in cluster varieties, Fock and Goncharov made the following
remarkable conjecture:

\begin{conjecture} \label{FGconj} $\cA^{\trop}(\bZ)$ parameterizes a canonical 
vector space basis
for $H^0(\cX,\cO_{\cX})$. The structure constants for the algebra 
$H^0(\cX,\cO_{\cX})$ expressed in this basis are non-negative integers.
\end{conjecture}

(Here we are treating the notationally simpler case of skew-symmetric cluster
varieties, the general case involving a Langlands dual seed.)

Note as stated $\cA$ and $\cX$ are on completely equal footing, so the conjecture
includes the analogous statement with the two reversed. 
Fock and Goncharov have a different definition of e.g., $\cA^{\trop}(\bZ)$,
which they denote $\cA(\bZ^t)$, as points of $\cA$ valued in the tropical semi-field. But it
is easy to check this agrees with our definition, which has the advantage that 
it makes
sense for any log CY, while theirs is restricted to varieties with a so-called
positive atlas of tori. 

In somewhat more detail, a skew-symmetric cluster variety is defined
using initial data of
a lattice $N$ with a skew-symmetric form $\{\cdot,\cdot\}:N\times N\rightarrow
\ZZ$, and each mutation
is given by the pair $(n,\{n,\cdot \})$ for some $n \in N$. When 
$\{\cdot,\cdot\}$ fails to be unimodular,
the dual $M$ does not have a skew-symmetric form, and in this case 
$\cA$ and $\cX$ are
on unequal footing. In this case $\cA$, by the Laurent phenomenon, 
always has lots of global functions, but $\cX$ may have very few.

Conjecture \ref{FGconj} was inspired by the case $\cA:= U \subset G$ of Example \ref{dbex}, 
which has a celebrated canonical basis of global functions 
constructed by G.\ Lusztig. See \cite{L90}.
Conjecture \ref{conjGHKI} suggests the existence of 
this basis may have nothing
a priori to do with representation theory, or cluster varieties, but is rather a general
feature of affine log CYs with maximal boundary. 

In \S\ref{FGcounterexamplessection}  
we show that Conjecture \ref{FGconj} as stated is often false. 
But if we add the condition
that $\cX$ is affine, it becomes a very special case of Conjecture 
\ref{conjGHKI}, and
for that reason we refer to $\cX, \cA$ as \emph{Fock-Goncharov mirrors}. 
In view of the highly involved existing proposals
for synthetic constructions of mirror varieties, \cite{KS06}, 
\cite{GSAnnals}, \cite{GHK11}, this simple alternative 
--- replace each torus in the open cover by its dual --- 
is an attractive surprise. 
We will prove many instances of Conjecture \ref{FGconj} in \cite{GHKK}. 

We now turn to the main idea in this paper, which connects the above
traditional description of cluster varieties via gluing tori to the
description we will develop in this paper, involving blowups of toric
varieties.  Here is some cluster motivation for the blowup approach. 
Each seed $\s$ gives
a torus open subset $T_{N,\s} \subset \cA$, together with $n$ cluster variables,
a basis of characters. These give a priori rational functions on $\cA$ and thus
a birational map $b:\cA \dasharrow \bA^n$, whose inverse restricts to an isomorphism
of the structure torus $\bG_m^n \subset \bA^n$ with $T_{N,\s} \subset \cA$. The Laurent
Phenomenon is equivalent to the statement that $b$ is regular, and thus in particular
suggests that each seed determines a construction of $\cA$ as (an open subset of) 
a blowup of a toric variety (in fact $\bA^n$) along a locus in the toric boundary. Stated this way,
it is natural to wonder if it holds for $\cX$ as well. We'll show this indeed holds
for $\cX$, and while it fails for general $\shA$, a slightly weaker
version is true which is still good enough for the Laurent Phenomenon.

Log CYs with maximal boundary are closed under blowup in the following sense:

\begin{lemma} Let $\bar U \subset \oY$ be a log CY open subset of a smooth 
(not necessarily
complete) variety, $\oD:= \oY \setminus \bar U$,
and $H \subset \oD\setminus\Sing(\oD)$ be a smooth codimension two 
(not necessarily irreducible) subvariety.
Let
$b:Y \to \oY$ be the blowup along $H$, 
$D \subset Y$ the strict transform
of $\oD$ and $U := Y \setminus D$. Then $U$ is log CY, with unique volume form the pullback
under $b$ of the volume form on $\bar U$. In
addition, $U$ has maximal boundary if $\bar U$ does.
\end{lemma}

\begin{proof} If $E$ is the exceptional divisor of $b$, then it
is standard that $K_Y=b^* K_{\oY}+E$ (using $H$ codimension two) and
that $D=b^*\oD-E$. Thus $K_Y+D=0$.
\end{proof}

Now starting with the simplest example, an algebraic torus, we get lots of examples via:

\begin{definition} \label{ccydef} Continuing with notation as in the lemma, 
we say that $U=Y\setminus D$ is a \emph{cluster log  CY} and 
$b:(Y,D) \to (\oY,\oD)$ a 
\emph{toric model} for $U$ if 
\begin{enumerate}
\item $(\oY,\oD)$ is toric and the fan for
$\oY$ consists only of one-dimensional cones 
$\RR_{\ge 0}v_i$ for $v_i\in N$ primitive, with $T_N$ the structure torus
of $\oY$. (Equivalently, the boundary $D$ is a disjoint
union of codimension one tori). 
\item 
The connected components of $H$ are the subtori $z^{w_i} + 1 = 0 
\subset T_{N/\bZ \cdot v_i}$ for
some $w_i \in (N/\bZ \cdot v_i)^* = v_i^{\perp} \subset M$. 
\end{enumerate}
\end{definition}

As the name suggests, the log CYs obtained by this simple blowup construction
and those
obtained in the previous discussion as tori glued in the simplest way are 
frequently the 
same. Note the toric model determines a canonical torus open subset 
\[
T_N = \oY \setminus \oD \subset Y \setminus D = U.
\]
Remarkably there are (usually infinitely) many other torus open sets. 
Given a toric model for $U$, and a choice of a center, i.e.,
a connected component of
$H$, or equivalently, a choice of one of the primitive lattice points $v = v_k$, there
is a natural {\it mutation} which produces a new log CY $U'$, 
with a birational
map $U \dasharrow U'$. Under certain conditions, this map
will be an isomorphism outside of codimension two subset (of
domain and range). In these nice situations, this produces, 
up to codimension two, a second copy 
of $T_N$ living in $U$. Iterating the procedure produces an atlas of torus open sets.    
Here is a sketch; full details are given in \S \ref{gcvsec}.

The connection with mutation of seeds comes via the tropical set.
Note a mutation
$\mu: U \dasharrow V$ between log CY varieties 
canonically induces an isomorphism of tropical sets
\[
\mu^t: U^{\trop}(\ZZ) \to V^{\trop}(\ZZ),\quad  v \to v \circ \mu^*.
\]
For the mutation $\mu_{(n,m)}: T_N \dasharrow T_N$ of Equation \eqref{of}, 
one computes
\begin{equation} \label{mof}
\mu^t_{(n,m)}: N = T_N^{\trop}(\ZZ) \to T_N^{\trop}(\ZZ) 
= N, \quad \mu^t(n') = n' + 
[\langle m,n'\rangle]_- n
\end{equation}
where for a real number $r$, $[r]_- := \min(r,0)$.
This illustrates the general
fact that $\mu^t$ is piecewise linear but not linear (unless $\mu$ 
is an isomorphism). 
This explains the geometric origin of piecewise linear maps in the cluster 
theory (and tropical geometry, see \cite{HKT09}, \S 2).
Here we view $U^{\trop}(\ZZ)$ as a collection of valuations. If we think of 
elements of $U^{\trop}(\ZZ)$ as boundary
divisors with integer weight, as in the second formula in equation \eqref{uteq},
$\mu^t$ is simply strict transform (also called pushforward) for the 
birational map $\mu$.  

Now we explain how to mutate from one toric model of a cluster log CY to another.
Continuing with the situation of Definition \ref{ccydef},
we choose one index, $k$, and let $v = v_k$, with corresponding divisor $D_k$.
The center $H_k=H\cap D_k$ 
determines what is known as an elementary 
transformation in algebraic geometry. We explain this in a simplified, but
key, situation. 

Let $\Sigma_v$ be the fan, with two rays, 
with support $\RR v$, so that the corresponding toric variety
$X_{\Sigma_v}\cong T_{N/\ZZ v}\times\PP^1$, with $\pi:X_{\Sigma_v}
\rightarrow T_{N/\ZZ v}$ the projection. Write $D_{\pm}$ for the two
toric divisors corresponding to the rays generated by $\pm v$. Viewing
$X_{\Sigma_v}\setminus D_-$ as an open subset of $\oY$, the center $H_k$
is identified with a codimension two subscheme $H_+\subset D_+\subset
X_{\Sigma_v}$. Let $H_-=\pi^{-1}(\pi(H_+))\cap D_-$.

There is then a birational map  $\mu: X_{\Sigma_v} \dasharrow X_{\Sigma_v}$
obtained by blowing up $H_+$ and then
blowing down the strict transform of $\pi^{-1}(\pi(H_+))$. One checks that 
$\mu$ is described
by Equation \eqref{of}. Clearly by construction $\mu$ is 
resolved by the blowup $b:Y' \to X_{\Sigma_v}$ along $H_+$, and one can check
that
$\mu \circ b: Y' \to X_{\Sigma_v}$ is regular as well, being the blowup
along $H_-$, see Lemma \ref{elemtransformlemma}.

This description of the elementary transformation extends to give
birational maps between closely related toric models. 
For simplicity assume $-v \neq v_i$
for any $i$ (in \S \ref{gcvsec} we consider the general case). 
Now let $\Sigma_+$ be the
fan consisting of rays $\RR_{\geq 0} v_i$ together with $-\RR_{\geq 0} v$. The
toric model gives us a blowup $b : Y \to X_{\Sigma_+}$. (This is a slight
abuse of notation, because we added one ray, $-\RR_{\geq 0}v$. 
But note we do not blow up
along the new boundary divisor $D_- \subset X_{\Sigma_+}$, and in forming
$U$ we throw away
the strict transform of boundary divisors, so adding this ray does not
change $U$ at all). Let $\Sigma_-$ be the fan with rays 
$\RR_{\geq 0} \mu^t(v_i)$ together with $-\RR_{\geq 0} v=-\RR_{\ge 0} \mu^t(v)$.
In \S \ref{codimtwosection} we show that in good situations,
\[
b' := \mu\circ b: Y \to X_{\Sigma_{-}}
\]
is regular off a codimension two subset and give formulae for the centers, 
which again are of the cluster log CY sort. 
Thus the
elementary transformation induces a new toric model for $U$ (up to changes
in codimension two), and in particular a second torus open subset of $U$.
This recovers the standard definition of mutations for cluster algebras \cite{FZ02a}. 
From this perspective, each seed is interpreted as the data for a toric model
of the 
same (up to codimension two) cluster log CY. Note in the {\it mutated} toric model
$b': Y \to X_{\Sigma_{-}}$ there is now a center in the boundary divisor $D_-$, but
no center in $D_+$. In the original model $b:Y \to X_{\Sigma_+}$ there is a
center in $D_+ \subset X_{\Sigma_+}$ (this divisor is the strict transform
of $D_+ \subset X_{\Sigma_{-}}$) but no center in  $D_- \subset X_{\Sigma_+}$. 
For all the other boundary divisors there is a center in either model. This 
difference between the chosen index $k$ and the other indices accounts for
the peculiar sign change in the formula for seed mutation, see 
Equation \eqref{etransform}.

Unfortunately, this procedure does not always give a precise identification
between the picture of cluster varieties as obtained from gluing of tori
and the picture given by blowups of toric varieties. The reason is that
$b'$ above need not always be regular off a codimension two subset. It
turns out that this works in certain cases, including all $\cX$ cluster
varieties and \emph{principal} $\cA$ cluster varieties. See \S 2 for
review of the definitions of the latter, and \S 3 for further details.

\begin{remark} There is no need to restrict to the special centers 
of  Definition \ref{ccydef}, (2): one can consider the blowup of an
arbitrary hypersurface in each boundary divisor. An elementary transform
gives a mutation of a toric model in the same way, but the formulae for
how the centers change are more complicated. 
For a general center, we checked one obtains
the mutation formulae of \cite{LP12}. In this note 
we restrict our treatment to the cluster variety case, 
as it is simpler and sufficient for our applications.
\end{remark}

There are lots of formulae in the Fomin-Zelevinsky, Fock-Goncharov definitions
of cluster algebras, which we reproduce in the next section. But we note that 
only one, Equation \eqref{of}, is essential. This is
the birational mutation, $\mu$, between tori in the $\cA$-atlas. Its canonical
dual, arising from Equation \eqref{dequation}, 
gives the mutation for the Fock-Goncharov mirror, see Equations  \eqref{Amutationeq}
and \eqref{Xmutationeq} below. 
The formula for the change of seed, Equation \eqref{etransform}, comes 
from the tropicalisation, $\mu^t$, of the 
birational mutation, Equation \eqref{mof}. Note in Equation \eqref{etransform},
$e_i' = \mu^t(e_i)$ for $i \neq k$, $e_k' = \mu^t(-e_k) = -e_k$. This is the
peculiar {\it sign change} explained above. 

\section{Review of the $\cX$ and $\cA$ cluster varieties}
\label{clusterreviewsection}

We follow \cite{FG09}, with minor modifications.
We will fix once and for all in the discussion the following data, which
we will refer to as \emph{fixed\footnote{This terminology is not standard
in the cluster literature. Rather, what we call fixed data along with seed
data is referred to as seed data in the literature. We prefer to distinguish
the data which remains unchanged under mutation from the data which changes.} 
data}:
\begin{itemize}
\item  A lattice $N$ with a skew-symmetric bilinear form
\[
\{\cdot,\cdot\}:N\times N\rightarrow \QQ.
\]
\item 
An \emph{unfrozen sublattice} $N_{\uf}\subseteq N$, a saturated
sublattice of $N$. If $N_{\uf}=N$, we say the fixed data has no frozen
variables.
\item An index set $I$ with $|I|=\rank N$ and a subset $I_{\uf}\subseteq I$
with $|I_{\uf}|=\rank N_{\uf}$.
\item Positive integers $d_i$ for $i\in I$ with greatest common divisor $1$.
\item A sublattice $N^{\circ}\subseteq N$ of finite index such that
$\{N_{\uf}, N^{\circ}\}\subseteq \ZZ$, $\{N,N_{\uf}
\cap N^{\circ}\}\subseteq\ZZ$.
\item $M=\Hom(N,\ZZ)$, $M^{\circ}=\Hom(N^{\circ},\ZZ)$.
\end{itemize}

Given this fixed data, \emph{seed data} for this fixed data
is a labelled collection of elements of $N$
\[
{\bf s}:=(e_i\,|\,i\in I)
\]
such that
$\{e_i\,|\,i\in I\}$ is a basis of $N$, 
$\{e_i \,|\, i\in I_{\uf}\}$ a basis
for $N_{\uf}$, and $\{d_ie_i\,|\,i\in I\}$
is a basis for $N^{\circ}$.

A choice of seed data $\s$ defines a new
(non-skew-symmetric) bilinear form on $N$ by
\begin{align*}
[\cdot,\cdot]_{\s}:N\times N\rightarrow & \QQ\\
[e_i,e_j]_{\s}= \epsilon_{ij}:= {} & \{e_i,e_j\}d_j
\end{align*}
Note that $\epsilon_{ij}\in\ZZ$ as long as we don't have $i,j\in I\setminus 
I_{\uf}$.
We note this bilinear form depends on the seed. We drop the subscript
$\s$ if it is obvious from context.

\begin{remark}
Suppose we specify a basis $e_i, i\in I$ for a lattice $N$, $I_{\uf}\subseteq I$,
positive
integers $d_i$, and a matrix $\epsilon_{ij}$ satisfying 
\[
d_i\epsilon_{ij}=-d_j\epsilon_{ji}
\]
and $\epsilon_{ij}\in\ZZ$ provided we don't have $i,j\in I\setminus I_{\uf}$.
This data determines the data $N$, $N_{\uf}$, $N^{\circ}$, $\{\cdot,\cdot\}$,
etc. It will turn out that
$\epsilon_{ij}$ for $i,j\in I\setminus I_{\uf}$ does not affect the
schemes we construct, and it is standard
in the literature to just consider rectangular matrices 
$(\epsilon_{ij})_{i\in I_{\uf}, j\in I}$.
We wish however to
emphasize that the fixed data does not depend on the particular
choice of seed.
\end{remark}

Given a seed $\s$, we obtain a dual basis $\{e_i^*\}$ for $M$,
and a basis $\{f_i\}$ of $M^{\circ}$ given by
\[
f_i=d_i^{-1}e_i^*.
\]
We use the notation
\[
\langle\cdot,\cdot\rangle:N\times M^{\circ}\rightarrow\QQ
\]
for the canonical pairing given by evaluation. We also write for $i\in I_{\uf}$
\[
v_i:=\{e_i,\cdot\}\in M^{\circ}.
\]
We have two natural maps defined by $\{\cdot,\cdot\}$:
\begin{align*} 
p_1^*:N_{\uf}\rightarrow M^{\circ}\quad\quad &\quad\quad\quad\quad p_2^*:N\rightarrow 
M^{\circ}/N_{\uf}^{\perp}\\
N_{\uf}\ni n\mapsto (N^{\circ}\ni n'\mapsto \{n,n'\})&\quad\quad N\ni n
\mapsto (N_{\uf}\cap N^\circ\ni n'\mapsto \{n,n'\})
\end{align*}
For the future, let us choose 
a map 
\begin{equation}
\label{pstardef}
p^*:N\rightarrow M^{\circ}
\end{equation}
such that, (a) $p^*|_{N_{\uf}}=p_1^*$ and
(b) the composed map $N\rightarrow M^{\circ}/N_{\uf}^{\perp}$ agrees with
$p_2^*$. Different choices\footnote{We
note that \cite{FG09} gives an incorrect definition when $N_{\uf}\not=N$,
as the formula $p^*(n)=\{n,\cdot\}$ may not give a result in $M^{\circ}$.} 
of $p^*$
differ by a choice of map $N/N_{\uf}\rightarrow N_{\uf}^{\perp}$.

Given seed data $\s$, we can associate two tori
\[
\hbox{$\shX_\s=T_M=\Spec\kk[N]$ and $\shA_\s= T_{N^{\circ}}=
\Spec\kk[M^\circ]$.}
\]
We write $X_1,\ldots,X_n$ as coordinates on
$\shX_\s$ corresponding to the basis vectors $e_1,\ldots,e_n$, i.e., $X_i=z^{e_i}$, and
similarly coordinates $A_1,\ldots,A_n$ corresponding to the basis vectors
$f_1,\ldots,f_n$, i.e., $A_i=z^{f_i}$.
The coordinates $X_i$, $A_i$ are called \emph{cluster variables}.
These coordinates give identifications 
\begin{equation}
\label{splitidentification}
\shX_\s\rightarrow \Gm^n, \quad \shA_\s\rightarrow\Gm^n.
\end{equation}
We write these two split tori as $(\Gm^n)_X$ and $(\Gm^n)_A$ in the 
$\cX$ and $\cA$ cases respectively.

\begin{remark}
\label{seedtoriremarks}
These tori come with the following structures:
\begin{enumerate}
\item Let $K=\ker p_2^*$. Then the inclusion $K\subseteq N$
induces a map $\shX_{\s}\rightarrow T_{K^*}=\Spec\kk[K]$.
Furthermore, the torus $T_{(N/N_{\uf})^*}=\Spec \kk[N/N_{\uf}]$ is
a subtorus of $\shX_{\s}$ and hence acts on $\shX_{\s}$.
\item Let $K^{\circ}=K\cap N^{\circ}$. Then
the inclusion $K^{\circ}\rightarrow N^\circ$ induces a map of tori
$T_{K^{\circ}}\rightarrow \shA_\s$.
This gives an action of $T_{K^{\circ}}$ on
$\shA_\s$. Furthermore, there is a natural inclusion $N_{\uf}^{\perp}
=\{m\in M^{\circ}\,|\, \langle m,n\rangle=0\quad
\forall n\in N_{\uf}\}\subseteq M^{\circ}$. This induces a map
$\shA_{\s}\rightarrow T_{N^{\circ}/N_{\uf}\cap N^{\circ}}=\Spec \kk[N_{\uf}^{\perp}]$.
\item 
The chosen map $p^*:N\rightarrow M^{\circ}$ defines a map
\[
p:\shA_{\s}\rightarrow \shX_\s.
\]
Furthermore, $p^*$ induces maps $p^*:K\rightarrow N_{\uf}^{\perp}\subseteq M^\circ$
and $p^*:N/N_{\uf}\rightarrow (K^{\circ})^*$, 
giving maps 
\[
p:T_{N^{\circ}/N_{\uf}\cap N^{\circ}}\rightarrow  T_{K^*},\quad\quad
p:T_{K^{\circ}}\rightarrow  T_{(N/N_{\uf})^*},
\]
respectively.
We then obtain commutative diagrams
\[
\hbox{
$\xymatrix@C=30pt{\shA_{\s}\ar[r]^p\ar[d]&\shX_{\s}\ar[d]\\
T_{N^{\circ}/N_{\uf}\cap N^{\circ}}\ar[r]_p&T_{K^*}}$
$\quad\quad\xymatrix@C=30pt{
T_{K^{\circ}}\ar[r]^p\ar[d]&T_{(N/N_{\uf})^*}\ar[d]\\
\shA_{\s}\ar[r]_p&\shX_{\s}}$
}
\]
\end{enumerate}
\end{remark}

We next define a mutation of seed data.

For $r\in\QQ$ define $[r]_+=\max(0,r)$. Given seed data 
${\bf s}$
and $k\in I_{\uf}$, we have a mutation $\mu_k(\s)$ of ${\bf s}$ given by
a new basis
\begin{equation}
\label{etransform}
e_i':=\begin{cases} e_i+[\epsilon_{ik}]_+e_k&i\not=k\\
-e_k&i=k.
\end{cases}
\end{equation}
Note that $\{e_i'\,|\,i\in I_{\uf}\}$ still form a basis for $N_{\uf}$ and the
$d_ie_i'$ still form a basis for $N^{\circ}$.
Dually, one checks that the basis $\{f_i\}$ for $M^{\circ}$ changes
as 
\[
f_i':=\begin{cases} -f_k+\sum_j [-\epsilon_{kj}]_+f_j& i=k\\
f_i&i\not=k.
\end{cases}
\]
One also checks that the matrix $\epsilon_{ij}$ changes via the formula
\begin{equation}
\label{epsmuteq}
\epsilon_{ij}':=\{e_i',e_j'\}d_j=\begin{cases}
-\epsilon_{ij} & k\in \{i,j\}\\
\epsilon_{ij} & \epsilon_{ik}\epsilon_{kj}\le 0, \quad k\not\in \{i,j\},\\
\epsilon_{ij}+|\epsilon_{ik}|\epsilon_{kj} & 
\epsilon_{ik}\epsilon_{kj} > 0, \quad k\not\in \{i,j\}.
\end{cases}
\end{equation}

We also define birational maps
\begin{align*}
\mu_k:\shX_\s\dasharrow & \shX_{\mu_k(\s)}\\
\mu_k:\shA_\s\dasharrow & \shA_{\mu_k(\s)}
\end{align*}
defined via pull-back of functions
\begin{align}
\label{Xmutationeq}
\mu_k^*z^n= {} & z^n(1+z^{e_k})^{-[n,e_k]},\quad n\in N\\
\label{Amutationeq}
\mu_k^*z^m= {} & z^m(1+z^{v_k})^{-\langle d_ke_k,m\rangle}, \quad m\in 
M^{\circ}.
\end{align}
These maps are more often seen in the cluster literature as described via
pull-backs of cluster variables:
\begin{equation}
\label{Xmutation}
\mu_k^*X_i'=\begin{cases} X_k^{-1} & i=k\\
X_i(1+X_k^{-\sgn(\epsilon_{ik})})^{-\epsilon_{ik}}& i\not=k
\end{cases}
\end{equation}
and
\begin{equation}
\label{Amutation}
A_k\cdot \mu_k^*A_k'=\prod_{j:\epsilon_{kj}>0}A_j^{\epsilon_{kj}}
+\prod_{j:\epsilon_{kj}<0} A_j^{-\epsilon_{kj}},\quad
\mu_k^*A_i'=A_i,\quad i\not = k.
\end{equation}
The correspondence between these two descriptions can be seen using
$X_i=z^{e_i}$, $X_i'=z^{e_i'}$ and $A_i=z^{f_i}$, $A_i'=z^{f_i'}$.

\begin{remark}
\label{tropicalmutremark}
Note in the notation of Equation \eqref{of}, the mutation
\eqref{Amutationeq} is
$\mu_{(-d_k e_k,v_k)}: T_{N^\circ} \dasharrow T_{N^\circ}.$ By Equation \eqref{mof} its tropicalisation is
\[
\mu_k^t(n) = n + [\langle v_k,n\rangle ]_{-}(-d_k e_k) = 
n + [\{n,d_ke_k\}]_+ e_k 
\]
and thus the seed mutation \eqref{etransform} is also given by
\begin{equation}
\label{Aseedtrop}
e_i'=\begin{cases} \mu_k^t(e_i) &i\not=k\\
-e_k = -\mu_k^t(e_k) &i=k.
\end{cases}
\end{equation}

On the other hand, the mutation \eqref{Xmutationeq} is
\[
\mu_{(d_kv_k,e_k)}: T_{M} \dasharrow T_{M}.
\]
This tropicalizes to
\[
\mu_k^t(m) = m + [\langle d_ke_k,m\rangle]_{-} v_k.
\]
Noting that as $p^*$ is a linear function, the $v_i$ transform under
the mutation in the same way the $e_i$ do, i.e., $v_k'=-v_k$, $v_i'=
v_i+[\epsilon_{ik}]_+ v_k$ for $i\not = k$. But 
\[
\mu_k^t(v_i)=v_i+[\epsilon_{ik}]_- v_k\not= v_i',
\]
so we do not obtain an equation analogous to \eqref{Aseedtrop}.
Rather, one checks that
\begin{equation}
\label{Xseedtrop}
-v_i'=\begin{cases} \mu_k^t(-v_i) & i\not=k\\
-\mu_k^t(-v_k)& i=k
\end{cases}
\end{equation}
\qed
\end{remark}

One checks easily the commutativity of the diagrams
\begin{equation}
\label{commdiag1}
\xymatrix@C=30pt
{T_{K^{\circ}}\ar[r]\ar[d]^=&\shA_{\s}\ar[r]^p\ar@![d]^{\mu_k}&\shX_\s\ar[r]
\ar@![d]^{\mu_k}&T_{K^*}\ar[d]^=\\
T_{K^{\circ}}\ar[r]&\shA_{\mu_k(\s)}\ar[r]_p&\shX_{\mu_k(\s)}\ar[r]&
T_{K^*}}
\end{equation}
\begin{equation}
\label{commdiag2}
\hbox{
$\xymatrix@C=30pt{
T_{(N/N_{\uf})^*}\ar[r]\ar[d]^=&\shX_{\s}\ar@![d]^{\mu_k}\\
T_{(N/N_{\uf})^*}\ar[r]&\shX_{\mu_k(\s)}}\quad\quad$
$\xymatrix@C=30pt{
\shA_{\s}\ar[r]\ar@![d]^{\mu_k}&T_{N^{\circ}/N_{\uf}\cap N^{\circ}}\ar[d]^=\\
\shA_{\mu_k(\s)}\ar[r]&T_{N^{\circ}/N_{\uf}\cap N^{\circ}}}$}
\end{equation}

We can now define the $\cX$ and $\cA$ 
cluster varieties associated to the seed $\s$. We will first need the following
general gluing construction:

\begin{proposition}
\label{gluingconstruction}
Let $\{X_i\}$ be a collection of integral, separated schemes of finite type
 over a field
$\kk$, with birational maps $f_{ij}:X_i \dasharrow X_j$ for all $i,j$,
with $f_{ii}$ the identity and $f_{jk}\circ f_{ij}=f_{ik}$ as rational
maps. Let $U_{ij}\subseteq X_i$ be the largest open subset such that
$f_{ij}:U_{ij}\rightarrow f_{ij}(U_{ij})$ is an isomorphism. Then there
is a scheme $X$ obtained by gluing the $X_i$ along the open sets $U_{ij}$
via the maps $f_{ij}$.
\end{proposition}

\begin{proof} 
First, the sets $U_{ij}$ exist: take $U_{ij}$ to consist of all
points $x$ in the domain of $f_{ij}$ at which $f_{ij}$ is a local isomorphism.
By \cite{Gr60}, 6.5.4, these are precisely the points $x$ such that
$f_{ij}^*:\shO_{X_j,f_{ij}(x)}\rightarrow \shO_{X_i,x}$ is an isomorphism.
By \cite{Gr60}, 8.2.8, $f_{ij}|_{U_{ij}}$ is an open immersion.

By \cite{H77}, Ex.\ II 2.12, it is now sufficient to
check that $f_{ij}(U_{ij}\cap U_{ik})=U_{ji}\cap U_{jk}$. Clearly
$f_{ij}(U_{ij}\cap U_{ik})\subseteq U_{ji}$. If $x\in U_{ij}\cap U_{ik}$,
then $f_{jk}$ can be defined at $f_{ij}(x)\in U_{ji}$ via 
$f_{ik}\circ f_{ij}^{-1}$. Then clearly $f_{jk}$ is a local isomorphism
at $f_{ij}(x)$, so $f_{ij}(x)\in U_{jk}$. Conversely, if $y\in U_{ji}
\cap U_{jk}$, then $y=f_{ij}(x)$ for some $x\in U_{ij}$. Clearly
$f_{ik}=f_{jk}\circ f_{ij}$ is a local isomorphism at $x$, so $x\in U_{ik}$
also and $y\in f_{ij}(U_{ij}\cap U_{ik})$.
\end{proof}

Let $\foT$ be the oriented rooted tree with $|I_{\uf}|$ outgoing edges from
each vertex, labelled by the elements of $I_{\uf}$. Let $v$ be the root of the 
tree. Attach the seed $\s$ to the vertex $v$. Now each simple path starting at
$v$ determines a sequence of seed mutations, just mutating at the label 
attached to the edge. In this way we attach a seed to each vertex of $\foT$.
We write the seed attached to a vertex $w$ as $\s_w$. We
further attach copies $\cX_{\s_w},\cA_{\s_w}$ to $w$. 

If $\foT$ has a directed edge from $w$ to $w'$ labelled with
$k\in I_{\uf}$, with
associated seeds $\s_w$ and $\mu_k(\s_w)=\s_{w'}$, we obtain mutations
$\mu_k:\cX_{\s_w}\dasharrow \cX_{\s_{w'}}$, $\mu_k:\cA_{\s_w}
\dasharrow \cA_{\s_{w'}}$. 
We can view these maps as arising from traversing the edge
in the direction from $w$ to $w'$; we use the inverse maps $\mu_k^{-1}$ 
if we traverse the edge from $w'$ to $w$.

Now for any two vertices $w,w'$ 
of $\foT$ there is a unique simple path $\gamma$ from one to 
the other. We obtain birational maps
\[
\mu_{w,w'}:\shA_{\s_w}\dasharrow \shA_{\s_{w'}},\quad
\mu_{w,w'}:\shX_{\s_w}\dasharrow \shX_{\s_{w'}},
\]
between the associated tori. These are obtained by taking the composition of
mutations or their inverses associated to each edge traversed by $\gamma$
in the order traversed, using a mutation $\mu_k$ associated to the
edge if the edge is traversed in the direction of its orientation, and
using $\mu_k^{-1}$ if traversed in the opposite direction.

These birational maps clearly satisfy $\mu_{w',w''}\circ \mu_{w,w'}=
\mu_{w,w''}$ as birational maps, and hence by Proposition 
\ref{gluingconstruction}, we obtain schemes
$\cX$ or $\cA$ by gluing these tori using these birational maps.

\begin{remark} 
\label{doublemutationremark}
Note that $\mu_k\circ\mu_k:\shA_{\s}\dasharrow \shA_{\mu_k
(\mu_k(\s))}$ is not the identity when expressed as a map 
$\Spec\kk[M^{\circ}]\dasharrow \Spec\kk[M^{\circ}]$; rather, it is the
isomorphism given by the linear map $M^{\circ}\rightarrow M^{\circ}$,
$m\mapsto m-\langle d_ke_k,m\rangle v_k$. This map takes the basis $\{f_i\}$ 
for the seed $\mu_k(\mu_k(\s))$ to the basis $\{f_i\}$ for the seed $\s$.
This is why $\mu_k\circ\mu_k$ is only the identity when viewed as an
automorphism of $\Spec \kk[A_1^{\pm 1}, \ldots, A_n^{\pm 1}]$.
\end{remark}

\begin{remark}
\label{Noetherianremark}
As we shall see in Theorem \ref{separatedtheorem}, 
the $\cA$ variety is always separated, but the $\cX$
variety usually is not. It is not clear, however, whether either of these
schemes is Noetherian. This will sometimes cause problems in what follows,
but these problems are purely technical. In particular, given any
finite connected regular subtree $\foT'$ of $\foT$, we can use the seed tori
corresponding to vertices in $\foT'$ to define open subschemes of
$\cX$ and $\cA$. We shall write these subschemes as $\cX^{\ft}$ and
$\cA^{\ft}$ respectively. We will not need to be particularly concerned
about which subtree $\foT'$ we use, only that it be sufficiently big
for the purpose at hand. However, we shall always assume $\foT'$
contains the root vertex $v$ and all its adjacent vertices.
\end{remark}

\begin{remark}
\label{extrastructuresremark}
The structures (1)-(4) of Remark \ref{seedtoriremarks}
described on individual seed tori, being
compatible with mutations as seen in Equations \eqref{commdiag1} and 
\eqref{commdiag2}, induce corresponding structure on $\shX$ and
$\shA$. In particular, 
(1) there is a canonical map 
\[
\lambda:\shX\rightarrow T_{K^*}
\]
and a canonical action of $T_{(N/N_{\uf})^*}$ on $\shX$;
(2) there is a canonical action of $T_{K^{\circ}}$ on $\shA$ and a canonical
map 
\[
\shA\rightarrow T_{N^{\circ}/N_{\uf}\cap N^{\circ}};
\]
(3) there is a map 
\[
p:\shA\rightarrow\shX.
\]
This map is compatible with the maps and actions of (1) and (2)
as indicated in Remark \ref{seedtoriremarks}, (3).
\end{remark}

\begin{definition} The $\shX$-cluster algebra ($\shA$-cluster algebra)
associated to a seed $\s$ is $\Gamma(\shX,\shO_{\shX})$
(or $\Gamma(\shA,\shO_{\shA})$).
\end{definition}

\begin{remark} The $\shA$-cluster algebra is usually called the 
\emph{upper cluster algebra} in the literature, see \cite{BFZ05}. This can
be viewed as the algebra of Laurent polynomials in $\kk[M^{\circ}]$
which remain Laurent polynomials under any sequence of mutations. Such a 
Laurent polynomial is called a \emph{universal Laurent polynomial}. The
algebra which is usually just called the \emph{cluster algebra} 
is the sub-algebra of the field of fractions 
$k(\shA_\s)=\kk(A_1,\ldots,A_n)$ of $\shA_\s$ generated by all functions 
\[
\{ \mu_{v,w}^*(A_i')\,|\, \hbox{$A_i'$ is a coordinate on $\shA_{\s_w}$,
$w$ a vertex of $\foT$}\}.
\]
We note that the cluster algebras arising via this construction are
still a special case of the general definition given in \cite{FZ02a},
and are called cluster algebras \emph{of geometric type}
in the literature. These
include most of the important examples.
\end{remark}

We end this section with several variants of the above constructions.

\begin{construction}
\label{frozenvariables}
When there are frozen variables (i.e., $N_{\uf}\not=N$) one frequently
might want to allow the frozen variables $X_i$, $i\not\in I_{\uf}$ or
$A_i$, $i\not\in I_{\uf}$ to take the value $0$. Thus one replaces
$\shX_{\s}$, $\shA_{\s}$ with 
\begin{align*}
\shX_{\s}:= {} & \Spec \kk[\{X_i^{\pm 1}\,|\,i\in I_{\uf}\}\cup
\{X_i\,|\,i\not\in I_{\uf}\}],\\
\shA_{\s}:= {} & \Spec \kk[\{A_i^{\pm 1}\,|\,i\in I_{\uf}\}\cup
\{A_i\,|\,i\not\in I_{\uf}\}].
\end{align*}
These varieties can be defined somewhat more abstractly as toric
varieties, with fans the set of faces of the cone generated by 
$\{e_i^*\,|\,i\not\in I_{\uf}\}$ and $\{d_ie_i\,|\,i\not\in I_{\uf}\}$ 
respectively. 
One sees from \eqref{Xmutation} and \eqref{Amutation} that no $X_i$ or
$A_i$ for $i\not\in I_{\uf}$ is inverted by mutations. Thus cluster varieties
$\shX$, $\shA$ can be defined via gluing these modified spaces as before.
In particular, we obtain a map $\shA\rightarrow \Spec\kk[\{A_i\,|\,i\not\in
I_{\uf}\}]$.

In any event, Fock and Goncharov \cite{FG11} 
define the \emph{special completion}
of the $\shX$ variety, written as $\widehat{\shX}$, by replacing each
$\shX_{\s}$ with the affine space $\Spec\kk[X_1,\ldots,X_n]$, and using
the same definition for the birational maps between the $\shX_{\s}$
as usual.
\end{construction}

\begin{construction}
\label{Aprincipaldef}
We define the notion of cluster algebra with \emph{principal 
coefficients}.
In general, given fixed data $N, \{\cdot,\cdot\}$ as usual 
along with seed data $\s$, we construct the \emph{double}
of the lattice via
\[
\widetilde{N}=N\oplus M^{\circ}, \quad
\{(n_1,m_1),(n_2,m_2)\}=\{n_1,n_2\}+\langle n_1,m_2\rangle-\langle n_2,m_1
\rangle.
\]
We take $\widetilde N_{\uf}=N_{\uf}\subseteq \widetilde N$, and $\widetilde 
N^\circ$
the sublattice $N^{\circ}\oplus M$.
The lattice $\widetilde{N}$ with its pairing $\{\cdot,\cdot\}$ and
sublattices $\widetilde{N}_{\uf}$, $\widetilde{N}^{\circ}$
can now play the role of fixed data. Given a seed $\s$ for the original fixed
data, we obtain a seed $\tilde\s$ for $\widetilde{N}$ with basis 
$\{(e_i,0), (0,f_{\alpha})\}$.
We use the convention that indices
$i,j,k\in I$ are used to index the first set of basis elements and
$\alpha,\beta,\gamma\in I$ are used to index the second set of basis elements.
The integer $d_i$ associated with $(e_i,0)$ or $d_{\alpha}$ associated to
$(0,f_{\alpha})$ is then taken to agree with $d_i$ or $d_{\alpha}$ of 
the original seed. Then
the matrix $\tilde\epsilon$ determined by this seed is given by
\[
\tilde\epsilon_{ij}=\epsilon_{ij}, \quad
\tilde\epsilon_{i\beta}=\delta_{i\beta},\quad \tilde\epsilon_{\alpha j}
=-\delta_{\alpha j}, \quad \tilde\epsilon_{\alpha\beta}=0.
\]
One notes that $\widetilde{M}=\Hom(\widetilde{N},\ZZ)=M\oplus N^{\circ}$
and $\widetilde{M}^{\circ}=M^{\circ}\oplus N$. Furthermore,
given a choice of $p^*:N\rightarrow M^{\circ}$, we can take
the map 
$p^*:\widetilde{N}\rightarrow \widetilde{M}^{\circ}$ to be given by
\[
p^*(e_i,0)=(p^*(e_i),e_i), \quad p^*(0,f_{\alpha})=(-f_{\alpha},0),
\]
so that $p^*$ is an isomorphism. 

With this choice of fixed and seed data, the corresponding 
$\shA$ cluster variety will be written as $\shA_{\prin}$. The ring
of global functions on $\shA_{\prin}$ is the upper cluster algebra
with principal coefficients at the seed $\s$ 
of \cite{FZ07}, Def.\ 3.1.

$\shA_{\prin}$ has an additional relationship with $\shX$. There are two natural
inclusions
\begin{eqnarray*}
\tilde p^*:N \rightarrow & \widetilde M^{\circ},
\quad\quad\quad \pi^*:&N\rightarrow
\widetilde M^{\circ}\\
n \mapsto & (p^*(n),n),\quad\quad\quad &n\mapsto (0,n)
\end{eqnarray*}
The first inclusion induces for any seed $\s$ an exact sequence of tori
\[
1\mapright{} T_{N^{\circ}}\mapright{} \shA_{\prin,\s}\mapright{\tilde p}
\shX_{\s}\mapright{} 1.
\]
One checks that $\tilde p$ commutes with the mutations $\mu_k$ on 
$\shA_{\prin,\s}$ and $\shX_{\s}$. Thus we obtain a morphism
$\tilde p:\shA_{\prin}\rightarrow\shX$. The $T_{N^{\circ}}$ action on 
$\shA_{\prin,\s}$ gives a $T_{N^{\circ}}$ action on $\shA_{\prin}$,
making $\tilde p$ the quotient map for this action and $\shA_{\prin}$
is a $T_{N^{\circ}}$-torsor over $\shX$. On the other hand, $\pi^*$
induces a projection
\begin{equation}
\label{Aprincipalmap}
\pi:\shA_{\prin}\rightarrow T_M.
\end{equation}
We note that if $e\in T_M$ denotes the identity element, then 
$\pi^{-1}(e)=\cA$. To see this, note the fibre of 
$\pi:\shA_{\prin,\s}\rightarrow
T_M$ over $e$ is canonically $\shA_\s$, and a mutation $\mu_k$ on
$\shA_{\prin,\s}$ specializes to the corresponding mutation on $\shA_\s$.
The open subset on which a mutation $\mu_{w,w'}:\shA_{\prin,\s_w}\rightarrow
\shA_{\prin, \s_{w'}}$ is an isomorphism onto its image restricts
to the corresponding open subset of $\shA_{\s_w}$; otherwise, $\shA_{\prin}$
would not be separated, contradicting Theorem \ref{separatedtheorem}.
\end{construction}

\begin{definition}
Let $t\in T_M$. We write $\shA_t$ for the fibre $\pi^{-1}(t)$. We call
this an $\cA$ \emph{cluster variety with general coefficients}.
\end{definition}

\begin{construction}
\label{variousrelations}
In case there are no frozen variables, i.e., $N=N_{\uf}$, we have $p^*=p_2^*$
and $K=\ker p^*$. We then have a commutative diagram
\[
\xymatrix@C=30pt
{N\ar[r]^{\tilde p^*}&\widetilde M^{\circ}\\
K\ar[u]^{\lambda^*}\ar[r]_{i^*}&N\ar[u]_{\pi^*}}
\]
where both $i^*$ and $\lambda^*$ are the inclusion. 
This induces a commutative diagram
\begin{equation}
\label{principaltoXdiagram}
\xymatrix@C=30pt
{\shX\ar[d]_{\lambda}&\ar[l]_{\tilde p}\shA_{\prin}\ar[d]^{\pi}\\
T_{K^*}&\ar[l]^{i}T_M
}
\end{equation}

Note that for $t\in T_M$, $\tilde p$ restricts to a map
\[
p_t:\shA_t\rightarrow \lambda^{-1}(i(t))=\shX_{i(t)}.
\]
\end{construction}

\section{The geometry of cluster varieties} \label{gcvsec}

We now give our description of cluster varieties as blowups of toric
varieties and mutations as elementary transformations of $\PP^1$-bundles.
This gives rise to most of the results in this paper, including
a simple explanation for the Laurent phenomenon and counterexamples to
some basic conjectures about cluster algebras.

\subsection{Elementary transformations.}
\label{elementarysection}

The basic point is that the gluing of adjacent seed tori can be easily
described in terms of blow-ups of toric varieties, and
that mutations have a simple interpretation
as a well-known operation in algebraic geometry known as an elementary
transformation. To describe this in general, we fix a lattice $N$
with no additional data, and a primitive vector $v\in N$. 
The projection $N\rightarrow N/\ZZ v$ gives a $\Gm$-bundle 
\[
\pi:T_N\rightarrow T_{N/\ZZ v}.
\]
A non-zero regular function $f$ on $T_{N/\ZZ v}$ can be viewed as a map 
\begin{align*}
f:T_{N/\ZZ v}\setminus V(f)\rightarrow {} & T_{\ZZ v}
\subseteq T_N=N\otimes_{\ZZ}\Gm\\
t\mapsto {} & v\otimes f(t)
\end{align*}
to obtain a birational map
\begin{align*}
\mu_f:T_N\dasharrow {}  & T_N\\
t \quad \mapsto {}  & f(\pi(t))^{-1}\cdot t.
\end{align*}
Note that on the level of pull-back of functions, this is defined, for
$m\in M=\Hom(N,\ZZ)$, by
\[
z^m\mapsto z^m (f\circ\pi)^{-\langle m,v\rangle}.
\]
Indeed, this is easily checked by choosing a basis $f_1,\ldots,f_n$ of
$M$ with $\langle f_1,v\rangle=1$, $\langle f_i,v\rangle=0$ for $i>1$.
This gives coordinates $x_i=z^{f_i}$, $1\le i \le n$,
on $T_N$ so that the projection $\pi$ is given by $(x_1,\ldots,x_n)
\mapsto (x_2,\ldots,x_n)$, and the map $\mu_f$ is given by
\[
(x_1,\ldots,x_n)\mapsto (f(x_2,\ldots,x_n)^{-1} x_1,x_2,\ldots,x_n).
\]

Now consider the fan $\Sigma_{v,+}=\{\RR_{\ge 0}v,0\}$
in $N$. This defines a toric variety $\TV(\Sigma_{v,+})$ isomorphic
to $\AA^1\times T_{N/\ZZ v}$, and contains a toric divisor $D_+$. It
has a canonical projection $\pi:\TV(\Sigma_{v,+})\rightarrow T_{N/\ZZ v}$,
which induces an isomorphism $D_+\cong T_{N/\ZZ v}$.
Set 
\[
Z_+=\pi^{-1}(V(f))\cap D_+.
\]
This hypersurface may be non-reduced. Define
\begin{align*}
&\hbox{$\widetilde{\TV}(\Sigma_{v,+})\rightarrow \TV(\Sigma_{v,+})$ the blowup
of $Z_+$,}\\
&\hbox{$\tilde D_+$ the proper transform of $D_+$,}\\
&U_{v,+}=\tTV(\Sigma_{v,+})\setminus \tilde D_+.
\end{align*} 
Note that $\Gamma(U_{v,+},\shO_{U_{v,+}})=\Gamma(\TV(\Sigma_{v,+}),
\shO_{\TV(\Sigma_{v,+})})[f/x_1]$.

We can also use $\mu_f$ to define a variety $X_f$ obtained by gluing
together two copies of $T_N$ using $\mu_f$ along the open subsets
$T_N\setminus V(f\circ \pi)\subseteq T_N$.

We then obtain the following basic model for describing gluings of
tori as blowups of toric varieties:

\begin{lemma}
\label{twotoriglue}
There is an open immersion $X_f\hookrightarrow U_{v,+}$
such that $U_{v,+}\setminus X_f$ is codimension
two in $U_{v,+}$. Furthermore, the projection $\pi:U_{v,+}
\rightarrow T_{N/\ZZ v}$ is a $\Gm$-bundle over $T_{N/\ZZ v}\setminus V(f)$,
while the fibres of $\pi$ over $V(f)$ are each a union of two copies of
$\AA^1$ meeting at a point. The locus where $\pi$ is not smooth is precisely
$U_{v,+}\setminus X_f$.
\end{lemma}

\begin{proof}
Using coordinates $(x_1,\ldots,x_n)$ for $\TV(\Sigma_{v,+})$ as before,
with $D_+$ given by $x_1=0$, note the ideal of $Z_+$ is $(x_1,f)$. 
Thus the blow-up of $Z_+$ is given by the equation $ux_1=vf$ in
$\PP^1\times \TV(\Sigma_{v,+})$. We define two embeddings of $T_N$,
\begin{align*}
\iota_1:(x_1,\ldots,x_n)\mapsto {} & \big((f,x_1),(x_1,\ldots,x_n)\big)\\
\iota_2:(x_1,\ldots,x_n)\mapsto {} & \big( (1,x_1), (x_1f,x_2,\ldots,x_n)\big)
\end{align*}
Noting that $\mu_f=\iota_2^{-1}\circ\iota_1$, 
it is clear that these maps give an embedding of $X_f$.
The divisor $\tilde D_+$ is given by the equation $v=x_1=0$,
so the only points of $U_{v,+}$ missed by the open immersion
$X_f\hookrightarrow U_{v,+}$ are the points where $u=x_1=0$, i.e.,
points of the form $\big((0,1),(0,x_2,\ldots,x_n)\big)$ with $f(x_2,
\ldots,x_n)=0$. The remaining statements are clear.
\end{proof}

Next we examine how this gives a basic model for a mutation. Consider the fan
$\Sigma_v:=\{\RR_{\ge 0}v, \RR_{\le 0}v, 0\}$.
This defines a toric variety we write as $\PP$, and it comes
with divisors $D_+$, $D_-$ corresponding to the two rays and a map
\[
\pi:\PP\rightarrow T_{N/\ZZ v},
\] 
identifying $D_+$ and $D_-$ with $T_{N/\ZZ v}$. Let 
\begin{align*}
Z_+= {} & 
D_+\cap V(f\circ \pi),\\
Z_-= {} & D_-\cap V(f\circ\pi).
\end{align*} We have two
blow-ups 
\[
b_{\pm}:\tilde\PP_{\pm} \rightarrow \PP
\]
being the blow-ups of $Z_{+}$ and $Z_{-}$.

\begin{lemma}
\label{elemtransformlemma}
The rational map $\mu_f:T_N\dasharrow T_N$ extends to a regular isomorphism
$\mu_f:\tilde\PP_+\rightarrow \tilde\PP_-$.
\end{lemma}

\begin{proof}
Working in coordinates $(x_1,\ldots,x_n)$ as before, we can
describe $\PP$ as $\PP^1\times T_{N/\ZZ v}$ with coordinates
$(x_1:y_1)$ on $\PP^1$ and coordinates $x_2,\ldots,x_n$ on
$T_{N/\ZZ v}$. Here $D_+$ is given by $x_1=0$ and $D_-$ by $y_1=0$.
Then $\mu_f$ is given as
\[
\big((x_1:y_1),(x_2,\ldots,x_n)\big)
\mapsto \big((x_1,f(x_2,\ldots,x_n)y_1),(x_2,\ldots,x_n)\big).
\]
This fails to be defined precisely where $x_1=f=0$, i.e., along
$Z_+$, and blowing up $Z_+$ clearly resolves this indeterminacy.
Thus $\mu_f:\PP\dasharrow\PP$ lifts to a morphism $\mu_f:\tilde\PP_+
\rightarrow\PP$. On the other hand, since the ideal sheaf  of
$Z_-$ in $\PP$ (locally generated by $y_1$ and $f$) pulls back via
$\mu_f$ to an invertible
sheaf on $\tilde\PP_+$, this morphism factors as a morphism 
$\mu_f:\tilde\PP_+\rightarrow\tilde\PP_-$ by the universal property
of blowing up.

To see that $\mu_f$ as viewed in this way is a regular isomorphism, note
the inverse rational
map $\mu_f^{-1}$ can be written as $t\mapsto f(\pi(t))\cdot t$, and
thus as a map $\PP\dasharrow\PP$ is written as
\[
\big((x_1:y_1),(x_2,\ldots,x_n)\big)
\mapsto \big((f(x_2,\ldots,x_n)x_1, y_1),(x_2,\ldots,x_n)\big).
\]
This lifts to a well-defined morphism $\mu_f^{-1}:\tilde\PP_-\rightarrow
\tilde\PP_+$ as before. Thus $\mu_f$ is an isomorphism
between $\tilde\PP_+$ and $\tilde\PP_-$.
\end{proof}

\begin{remark}
This lemma should be interpreted as saying
that $\mu_f:\PP\dasharrow\PP$ can be viewed
as the birational map described as the blow-up of $Z_+$ followed by
the contraction of the proper transform of $\pi^{-1}(V(f))\subseteq \PP$
in $\tilde\PP_+$ to $Z_-\subseteq \PP$. This is a birational operation called
an \emph{elementary transformation} in algebraic geometry.

Furthermore, 
let $\tilde D_{\pm}$ be the proper transform of $D_{\pm}$ in either
$\tilde\PP_+$ or $\tilde \PP_-$. Then 
combining Lemmas \ref{twotoriglue} and \ref{elemtransformlemma}, 
this tells us that there are open immersions of
$X_f$ in $\tilde\PP_{\pm}\setminus 
(\tilde D_+\cup \tilde D_-)$, missing a codimension
two subset. 
The roles the two coordinate tori of $X_f$ play are reversed under these
two immersions;
one of the tori of $X_f$ is the inverse image of the big torus orbit
under the blow-up $\tilde\PP_-\rightarrow\PP$, and the other torus in 
$X_f$ is the inverse image of the big torus orbit under the blow-up
$\tilde\PP_+\rightarrow\PP$.
\end{remark}

We need an extended version of the above setup: 

\begin{construction}
\label{basicconstruction}
Suppose we have the data of a fan $\Sigma=\{\RR_{\ge 0}v_i\,|\,
1\le i\le \ell\}\cup\{0\}$ where $v_1,\ldots,v_{\ell}\in N$ are 
primitive, $w_1,\ldots,w_\ell\in M$ with
$\langle v_i,w_i\rangle =0$. We allow some of the $v_i$'s to coincide.
Let $a_1,\ldots, a_{\ell}$ be positive integers,
$c_1,\ldots,c_{\ell}\in\kk^{\times}$, and $\mu_i:T_N\dasharrow T_N$
be defined as before by the data $f_i=(1+c_iz^{w_i})^{a_i}$ and $v_i$, where
$c_i\in \kk^{\times}$.
Let $\TV(\Sigma)$ be the toric variety defined by $\Sigma$, and let
$D_i$ be the toric divisor corresponding to $\RR_{\ge 0} v_i$. 

In what follows, we use the notation $\bar V(f_i)$ for the closure
of $V(f_i)\subseteq T_N$ in $\TV(\Sigma)$. Define
\begin{align*}
&Z_j=  D_j\cap \bar V(f_j),\\
&\pi:\tTV(\Sigma)\rightarrow \TV(\Sigma)\,\,\hbox{the blow-up
along $\bigcup_{i=1}^{\ell} Z_i$,}\\
&\tilde D_j  \,\,\hbox{the proper transform of $D_j$}.
\end{align*}

On the other hand, define a scheme $X$ as follows. Let $T_0,\ldots,T_{\ell}$
be $\ell+1$ copies of the torus $T_N$. The map $\mu_i$ is viewed as
an isomorphism between open sets 
\[
\varphi_{0i}:=\mu_i:U_{0i}\rightarrow U_{i0}
\]
of
$T_0$ and $T_i$ respectively, with $U_{0i}$ taken as the largest possible
such open subset. Indeed, we can take $U_{0i}=T_0\setminus V(f_i)$
and $U_{i0}=T_i\setminus V(f_i)$. In addition, for $1\le i,j\le \ell$, 
define $\varphi_{ij}:=\mu_j\circ\mu_i^{-1}$, and define $U_{ij}$ to be
the largest subset of $T_i$ on which $\varphi_{ij}$ defines an open
immersion.
The identifications $\varphi_{ij}$ then
provide gluing data to obtain a scheme $X$, in
general not separated, by Proposition \ref{gluingconstruction}.

\begin{lemma}
\label{torusgluinglemma2}
There is a natural morphism
\[
\psi:X\rightarrow \tilde U_{\Sigma}:=
\tTV(\Sigma)\setminus \bigcup_i \tilde D_i,
\]
which in special cases satisfies the following properties:
\begin{enumerate}
\item If $\dim Z_i\cap Z_j <\dim Z_i$ for all $i\not=j$, then 
$\psi$ is an isomorphism off a set of codimension $\ge 2$.
\item If $Z_i\cap Z_j=\emptyset$ for all $i\not=j$, then
$\psi$ is an open immersion. In particular, in this case, $X$
is separated.
\end{enumerate}
\end{lemma}

\begin{proof}
This is just a slightly more involved version of the argument
of Lemma \ref{twotoriglue}. We first describe maps of the tori
$T_i$, $0\le i \le \ell$ into $\tilde U_{\Sigma}$. We have a canonical
identification of
$T_0$ with the big torus orbit $T_N$ of $\TV(\Sigma)$,
isomorphic to $\pi^{-1}(T_N)\subseteq \tilde U_{\Sigma}$. On the other
hand, for a given $i$, let $J$ be the set of indices such that
$v_j=v_i$ if and only if $j\in J$. Note  $\TV(\Sigma_{v_i,+})$ is an open
subset of $\TV(\Sigma)$. Using coordinates $x_1,\ldots,x_n$ on
$\TV(\Sigma_{v_i,+})$ as in the proof of Lemma \ref{twotoriglue},
we obtain an open subset of $\tTV(\Sigma)$ described as a subset of
$\TV(\Sigma_{v_i,+})\times\PP^1$ given by the equation $ux_1=v\prod_{j\in J}f_j$.
With this description, we define $\iota_i:T_i\rightarrow \tilde U_{\Sigma}$ by
\[
\iota_i:(x_1,\ldots,x_n)\mapsto\big((\prod_{j\in J\setminus \{i\}} f_j,x_1),
(f_ix_1,x_2,\ldots,x_n)\big).
\]
Note that $\iota_i$ contracts the locus $f_i=\prod_{j\in J\setminus\{i\}}f_j
=0$
in $T_i$ so this is not an embedding unless the $Z_j$ are disjoint.
In this coordinate chart, $\iota_0$ is given by
\[
\iota_0:(x_1,\ldots,x_n)\mapsto \big((\prod_{j\in J} f_j,x_1),(x_1,\ldots,x_n)
\big).
\]
From this one sees that $\iota_i\circ\mu_{i}=\iota_0$ on $U_{0i}$. In particular
the maps $\iota_i$, $0\le i \le n$ are compatible with the gluings
$\varphi_{ij}$, and hence we obtain the desired map $\psi$.

In the case (1), each $\iota_i$, $i\ge 1$, is an open immersion off of a 
codimension $\ge 2$ set, and as in Lemma \ref{twotoriglue}, it is easy
to see the image misses a codimension $\ge 2$ set.
In case (2), each $\iota_i$ is an open immersion. Thus $\psi$ is
a local isomorphism, and it is enough to
show $\psi$ is injective to see that it is an open immersion. Certainly
$\psi$ is injective on each $T_i$. If $x\in T_i$, $y\in T_j$ have
$\psi(x)=\psi(y)$, then $\iota_i(x)=\iota_j(y)$. Noting that
$\varphi_{ij}=\iota_j^{-1}\circ\iota_i$ as rational maps, we see
that $\varphi_{ij}$ is a local isomorphism at $x$ and $\varphi_{ij}(x)
=y$. Thus $x\in U_{ij}$ and $x$ and $y$ are identified by the gluing maps
so they give the same point in $X$.
\end{proof}

Next we understand the general setup for a mutation.

Given elements $v\in N$, $w\in M$ with $\langle w,v\rangle=0$,
define the piecewise linear transformation
\[
T_{v,w}:N_{\RR}\rightarrow N_{\RR},\quad
n\mapsto  
n+[\langle n,w\rangle]_- v
\]
Note this coincides with the tropicalization of $\mu_{(v,w)}$ in 
\eqref{of} as given in \eqref{mof}.

Now in the situation of this construction, let us impose one additional
restriction on the starting data $v_i,w_i$, namely,
\begin{equation}
\label{mutationassumption}
\langle w_i,v_j\rangle=0 \Leftrightarrow \langle w_j,v_i\rangle=0.
\end{equation}
Pick some index $k$ and let
\[
\Sigma_+=\Sigma\cup \{\RR_{\le 0}v_k\},
\]
and define $\Sigma_-$ by applying $T_{-v_k,a_kw_k}$ to each ray of
$\Sigma_+$. Let $D_{k,+}\subseteq \TV(\Sigma_+)$
be the divisor corresponding to $\RR_{\ge 0} v_k$ in $\Sigma_+$
and $D_{k,-}\subseteq \TV(\Sigma_-)$ be the divisor corresponding
to $\RR_{\le 0}v_k$ in $\Sigma_-$. For $j\not=k$, write $D_{j,\pm}$
for the divisor corresponding to $\RR_{\ge 0} v_j$ in $\Sigma_+$
or $\RR_{\ge 0} T_{-v_k,a_kw_k}(v_j)$ in $\Sigma_-$. Finally, we can 
set
\begin{align*}
Z_{j,+}=&\bar V(f_j)\cap D_{j,+}\\
Z_{j,-}=&
\begin{cases}
\bar V(f_j)\cap D_{j,-}& \hbox{if $\langle w_k,v_j\rangle \ge 0$}\\
\bar V\big((1+c_jc_k^{a_k\langle w_j,v_k\rangle}
z^{w_j+a_k\langle w_j,v_k\rangle w_k})^{a_j}\big)\cap D_{j,-}&
\hbox{if $\langle w_k,v_j\rangle \le 0$}.
\end{cases}
\end{align*}
Let $\tTV(\Sigma_\pm)$ be the blowups of $\TV(\Sigma_{\pm})$ at
this collection of subschemes.

\begin{lemma} 
\label{elemtransformlemma2}
\[
\mu_k=\mu_{f_k}:T_N\dasharrow T_N
\]
defines a birational map
\[
\mu_{k}: \tTV(\Sigma_+)\dasharrow\tTV(\Sigma_-).
\]
If $\dim \bar V(f_k)\cap Z_{j,+} < \dim Z_{j,+}$ whenever
$\langle w_k,v_j\rangle =0$, then this extension
is an isomorphism off of sets of codimension $\ge 2$.
\end{lemma}

\begin{proof}
We first analyze the map $\mu_k$ before blowing up the hypersurfaces
$Z_{j,+}$, $j\not =k$. So abusing notation, assume $\tTV(\Sigma_\pm)$
is just obtained by blowing up $Z_{k,\pm}$.
Off of a closed subset of codimension two,
we can cover $\tTV(\Sigma_+)$ with open sets, one isomorphic to
$\tilde\PP_+$ with $v=v_k$, and the remaining ones of the form 
$U_\rho\setminus \bar V(f_k)$. 
Here $\rho$ ranges over dimension one cones of $\Sigma_+$
not equal to $\RR_{\ge 0}v_k$ or $\RR_{\le 0}v_k$, and $U_{\rho}$
denotes the standard affine toric open subset of $\TV(\Sigma_+)$
corresponding to $\rho$. Denoting $D_{\rho} \subseteq U_{\rho}$ the toric
divisor, note that $D_{\rho}\cap \bar V(f_k)=\emptyset$
if $w_k$ is non-zero on $\rho$, as then either $z^{w_k}$ or $z^{-w_k}$ vanishes
on $D_{\rho}$ and $\bar V(1+z^{w_k})=\bar V(1+z^{-w_k})$. Thus we only fail to cover
codimension two subsets of the form $D_{\rho}\cap \bar V(f_k)$ such
that $w_k$ is zero on $\rho$. So for the purposes of describing
the extension of $\mu_k$ up to codimension two, it will be sufficient to
restrict to the open subset $U$ of $\tTV(\Sigma_+)$
covered by these open sets.

By Lemma \ref{elemtransformlemma},
$\mu_k$ gives a well-defined morphism on the open subset isomorphic to
$\tilde\PP_+$, so we need to check $\mu_k$ defines a morphism on each
of the remaining sets. If $\langle w_k,\rho\rangle \ge 0$, then 
for any $m\in \rho^{\vee}\cap M=(T_{-v_k,a_kw_k}(\rho))^{\vee}\cap M$, 
$\mu_k^*$ acts by
\[
z^m\mapsto z^mf_k^{-\langle m,v_k\rangle},
\]
taking a regular function to a regular function on $U_{\rho}\setminus
\bar V(f_k)$. If $\langle w_k,\rho\rangle
< 0$, then if $m\in (T_{-v_k,a_kw_k}(\rho))^{\vee}\cap M$ we see $\mu_k$
acts by
\[
z^m\mapsto z^m(1+c_kz^{w_k})^{-a_k\langle m,v_k\rangle}
=z^{m-a_k\langle m,v_k\rangle w_k}(c_k+z^{-w_k})^{-a_k\langle m,v_k\rangle}.
\]
But $m-a_k\langle m,v_k\rangle w_k\in \rho^{\vee}$ by definition of 
$T_{-v_k,a_kw_k}$, so this is again a regular function on $U_{\rho}\setminus 
\bar V(f_k)$.
This shows $\mu_k$ is a morphism on $U$;
to show it is an isomorphism onto its image, we repeat
the same process for $\mu_k^{-1}$.

To prove the result after blowing up the hypersurfaces $Z_{j,\pm}$,
first note that if $\langle w_k,v_j\rangle\not=0$, then $Z_{j,+}\subseteq U$,
and we need to show that $\mu_k(Z_{j,+})=Z_{j,-}$. This can be checked
in cases. If $\langle w_k,v_j\rangle\ge 0$, then $Z_{j,-}$ is defined
by the equation $f_j$ on $D_j$. Now if $\langle w_k,v_j\rangle >0$,
we have $f_k|_{D_j}=1$, so that
$\mu_k^*(f_j)|_{D_j}=f_j|_{D_j}$. 
If $\langle w_k,v_j\rangle=0$, then $\langle w_j,v_k\rangle
=0$ by Assumption \eqref{mutationassumption}, so that $\mu_k^*z^{w_j}=
z^{w_j}$, so again $\mu_k^*(f_j)=f_j$. If $\langle w_k,v_j
\rangle < 0$, then noting the definition of $Z_{j,-}$ in this case,
\begin{align*}
 & \mu_k^*((1+c_jc_k^{a_k\langle w_j,v_k\rangle} z^{w_j+a_k\langle w_j,v_k\rangle w_k})^{a_j})\\ 
= {} & (1+c_jc_k^{a_k\langle w_j,v_k\rangle} z^{w_j+a_k\langle w_j,v_k\rangle w_k}(1+c_kz^{w_k})^{-a_k
\langle w_j,v_k\rangle})^{a_j} \\
= {} & (1+c_jc_k^{a_k\langle w_j,v_k\rangle} z^{w_j}(c_k+z^{-w_k})^{-a_k\langle w_j,v_k\rangle})^{a_j}.
\end{align*}
However, $z^{-w_k}$ vanishes identically on $D_j$ in this case, so
restricting to $D_j$ this coincides with $f_j$.
This shows $\mu_k$ extends to a regular map after blowing up $U$ along the
$Z_{j,\pm}$ for those $j$ with $\langle w_k,v_j\rangle\not=0$.

Finally, if $\langle w_k,v_j\rangle=0$, then we do not necessarily
have $Z_{j,+}\subseteq U$, and if $\bar V(f_k)$ contains an irreducible
component of $Z_{j,+}$, the map $\mu_k$ need not extend as an isomorphism
across the exceptional divisor of the blowup of $Z_{j,+}$. Hence we
need to use the stated hypothesis, which implies that $Z_{j,+}\setminus U$
is codimension $\ge 3$. Since $\mu_k^*(f_j)=f_j$ when $\langle w_k,v_j\rangle
=0$, it then follows that $\mu_k$ extends to an isomorphism off of a set of
codimension $\ge 2$ in $\tTV(\Sigma_+)$.
\end{proof}
\end{construction}

\subsection{The $\shX$- and $\shA_{\prin}$-cluster varieties up to codimension two.}
\label{codimtwosection}

Since the ring of functions on a non-singular variety is determined off
a set of codimension two, we can study the $\shX$- and $\shA_{\prin}$-cluster
algebras by describing the corresponding varieties up to codimension two.

Suppose given fixed data as in \S\ref{clusterreviewsection}.
Let $\s$ be a seed. Consider the fans
\begin{align*}
\Sigma_{\s,\shA}:={} & \{0\}\cup \{\RR_{\ge 0}d_ie_i\,|\,i\in I_{\uf}\}\\
\Sigma_{\s,\shX}:={} & \{0\}\cup \{-\RR_{\ge 0}d_iv_i\,|\,i\in I_{\uf}\} 
\end{align*}
in $N^{\circ}$ and $M$ respectively. These define toric varieties
$\TV_{\s,\shA}$ and $\TV_{\s,\shX}$ respectively. We remark that the minus
signs in the definition of $\Sigma_{\s,\shX}$ are forced on us by
\eqref{Xseedtrop}.

Each one-dimensional ray in one of these fans corresponds to a toric
divisor, which we write as $D_i$ for $i\in I_{\uf}$ (not distinguishing
the $\cX$ and $\cA$ cases). For $i\in I_{\uf}$, we can define closed subschemes
\begin{align}
\label{ZAXdef}
\begin{split}
Z_{\shA,i}:= {} & D_i \cap \bar V(1+z^{v_i})\subseteq \TV_{\s,\shA},\\
Z_{\shX,i}:= {} & D_i \cap \bar 
V((1+z^{e_i})^{\ind d_iv_i})\subseteq \TV_{\s,\shX},
\end{split}
\end{align}
where $\ind d_iv_i$ denotes the greatest degree of divisibility of $d_iv_i$ in
$M$.
Let $(\tTV_{\s,\shA},D)$ and $(\tTV_{\s,\shX},D)$ be the pairs consisting
of the blow-ups of $\TV_{\s,\shA}$ and $\TV_{\s,\shX}$ along the closed
subschemes $Z_{\shA,i}$ and $Z_{\shX,i}$ respectively, with $D$ the
proper transform of the toric boundaries. 

We note that in the $\shA$ case the divisors $D_i$ are distinct and hence
the centers of the blow-ups are disjoint. In the $\shX$ case, however,
we might have $v_{i}$ and $v_{i'}$ being positively proportional to
each other, so that $D_i=D_{i'}$. Then the two centers $Z_{\shX,i}$,
$Z_{\shX,i'}$ may intersect. However, it is easy to see this intersection
occurs in higher codimension, i.e., 
$\dim Z_{\shX,i}\cap Z_{\shX,i'} < \dim Z_{\shX,i}$. 
Thus in the $\cX$ case we are in the situation of Lemma 
\ref{torusgluinglemma2}, (1)
and in the $\cA$ case we are in the situation of Lemma 
\ref{torusgluinglemma2}, (2).

Finally we define
\[
U_{\s,\shA}:=\tTV_{\s,\shA}\setminus D,\quad
U_{\s,\shX}:=\tTV_{\s,\shX}\setminus D.
\]
Clearly these varieties contain the seed tori $\shA_\s$ and $\shX_s$,
and hence given vertices $w,w'\in\foT$, we obtain
a birational
map $\mu_{w,w'}$ of seed tori inducing birational maps 
\[
\mu_{w,w'}:U_{\s_w,\shA}\dasharrow U_{\s_{w'},\shA},\quad
\mu_{w,w'}:U_{\s_w,\shX}\dasharrow U_{\s_{w'},\shX}.
\]

Since $\shA_{\prin}$ is defined to be a special case of the construction
of the $\shA$ cluster variety, we also obtain in the same way birational
maps
\[
\mu_{w,w'}:U_{\s_w,\shA_{\prin}}\dasharrow U_{\s_{w'},\shA_{\prin}}.
\]
In this case the projection $\widetilde N^{\circ}\rightarrow M$ projects
all rays of $\Sigma_{\s,\shA_{\prin}}$ to $0$, so we obtain a morphism
$\TV_{\s, \shA_{\prin}}\rightarrow T_M$. The fibres of this map are
(non-canonically) isomorphic to $\TV_{\s,\shA}$. After blowing up the centers
$Z_{\shA_{\prin},i}$, we get morphisms $\pi:U_{\s,\shA_{\prin}}\rightarrow
T_M$ which commute with the mutations $\mu_{w,w'}$. Write a fibre of
$\pi$ over $t\in T_M$ as $U_{\s,\shA_t}$. We then obtain birational
maps on fibres of $\pi$ over $t$:
\[
\mu_{w,w'}:U_{\s_w,\shA_{t}}\dasharrow U_{\s_{w'},\shA_t}.
\]

We recall from \cite{BFZ05}:

\begin{definition}
A seed $\s$ is \emph{coprime} if, writing \eqref{Amutation} as $A_k\cdot
\mu_k^* A_k'=P_k$, the $P_k$, $k\in I_{\uf}$, are pairwise coprime.
We say a seed $\s$ is \emph{totally coprime} if all seeds obtained by
repeated mutations of $\s$ are coprime.
\end{definition}

We then have

\begin{lemma}
\label{keylemma}
Let $U'_{\s,\cA} \subset \cA$ (resp. $U'_{\s,\cX} \subset \cX$) be the union of the tori $\cA_{\s}$ (resp. $\cX_{\s}$)
and $\cA_{\mu_i(\s)}$ (resp. $\cX_{\mu_i(\s)}$), $i \in I_{\uf}$.

\begin{enumerate}
\item
For $k\in I_{\uf}$, with $w'=\mu_k(w)$, the maps
\[
\mu_{w,w'}:U_{\s_w,\shX}\dasharrow U_{\s_{w'},\shX},\quad
\mu_{w,w'}:U_{\s_w,\shA_{\prin}}\dasharrow U_{\s_{w'},\shA_{\prin}}
\]
are isomorphisms outside codimension two.
\item 
$\mu_{w,w'}:U_{\s_w,\shA}\dasharrow U_{\s_{w'},\shA}$
is an isomorphism outside codimension two if the seed $\s_w$ is coprime.
\item
$\mu_{w,w'}:U_{\s_w,\shA_t}\dasharrow U_{\s_{w'},\shA_t}$ is an isomorphism
outside of codimension two for $t\in T_M$ general (i.e., $t$ contained in
some non-empty Zariski open subset).
\item $U'_{\s,\cA} \dasharrow U_{\s,\cA}$ is an open immersion with image an open subset whose complement
has codimension at least two.
\item $U'_{\s,\cX} \dasharrow U_{\s,\cX}$ is an isomorphism outside of codimension two. 
\end{enumerate}
\end{lemma}

\begin{proof}
These are all special cases of Construction \ref{basicconstruction}. 
For (1) and (2), in
the $\cX$ (resp.\ $\cA_{\prin}$, $\cA$) case, we take the vectors
$v_i$ to be $-d_iv_i/\ind(d_iv_i)\in M$ 
(resp.\ $(d_ie_i,0)\in\widetilde N^{\circ}$,
$d_ie_i\in N^{\circ}$) for $i\in I_{\uf}$, the vectors $w_i$ to be $e_i\in N$
(resp.\ $(v_i,e_i)\in \widetilde M^{\circ}$, $v_i\in M^{\circ}$). In
all these cases, the constants $c_i$ are taken to be $1$. The integers
$a_i$ are taken to be $a_i=\ind(d_iv_i)$ (resp.\ $a_i=1$). In all three
cases, the cluster mutation $\mu_k$ coincides with the $\mu_k$
as defined in Construction \ref{basicconstruction}. In the notation of
Lemma \ref{elemtransformlemma2}, taking $\Sigma_+=\Sigma_{\s_w,\shX}$ 
(resp.\ $\Sigma_{\s_w,\shA_{\prin}}$, $\Sigma_{\s_w,\shA}$),
we observe that $T_{d_kv_k,e_k}$ (resp.\ 
$T_{(-d_ke_k,0),(v_k,e_k)}$,
$T_{-d_ke_k,v_k}$) applied to the rays of $\Sigma_+$
gives $\Sigma_-:=\Sigma_{\s_{w'},\shX}$, (resp.\
$\Sigma_{\s_{w'},\shA_{\prin}}$, $\Sigma_{\s_{w'},\shA}$) 
as follows immediately from \eqref{etransform} and Remark 
\ref{tropicalmutremark}.

We now only need to check the hypothesis of Lemma \ref{elemtransformlemma2}
to see that $\mu_{w,w'}$ is an isomorphism off codimension two
subsets. In the $\cX$ case, $f_k=1+z^{e_k}$, and from this the condition
is easily checked. In the $\cA$ case, $f_k=1+z^{v_k}$, which coincides with
$P_k$ up to a monomial factor. The hypothesis then follows from the coprime
condition, and the principal coefficient case is automatically coprime
as the $(v_k,e_k)$, $k\in I_{\uf}$ are linearly independent.

The $\shA_t$ case (3) is similar to the $\shA$ case, except that now
$f_k=1+z^{e_k}(t)\cdot z^{v_k}$, so we take $c_k=z^{e_k}(t)$. If $t$
is chosen generally, then the hypothesis of Lemma \ref{elemtransformlemma2}
continues to hold.

(5) follows from part (1) of Lemma \ref{torusgluinglemma2}. 
(4) follows from part (2) of  Lemma \ref{torusgluinglemma2}. 
\end{proof}

The main result in this section is then:

\begin{theorem}
\label{maingeometrictheorem}
Let $w,w'$ be vertices in $\foT$.
\begin{enumerate}
\item
The induced birational maps
\begin{align*}
U_{\s_w,\shX}\smash{\mathop{\dasharrow}\limits^{\mu_{w,w'}} }
U_{\s_{w'},\shX}\dasharrow& \shX^{\ft}\\
U_{\s_w,\shA_{\prin}}\smash{\mathop{\dasharrow}\limits^{\mu_{w,w'}} }
U_{\s_{w'},\shA_{\prin}}\dasharrow& \shA_{\prin}^{\ft}
\end{align*}
are isomorphisms outside of codimension two. (See Remark \ref{Noetherianremark}
for $\shX^{\ft}$, $\shA^{\ft}$. We use a finite subtree of $\foT$ containing
both $w$ and $w'$.) 
\item
If the initial seed is totally coprime, then
\[
U_{\s_w,\shA}\smash{\mathop{\dasharrow}\limits^{\mu_{w,w'}} }
U_{\s_{w'},\shA}\dasharrow \shA^{\ft}
\]
is an isomorphism outside a codimension two set. 
\item
If $t\in T_M$ is very general (outside a countable union of proper closed
subsets), then 
\[
U_{\s_w,\shA_t}\smash{\mathop{\dasharrow}\limits^{\mu_{w,w'}} }
U_{\s_{w'},\shA_t}\dasharrow \shA_t^{\ft}
\]
is an isomorphism outside a codimension two set.
\end{enumerate}
In particular, as
all schemes involved are $S_2$, these maps
induce isomorphisms on rings of regular functions.

\end{theorem}

\begin{proof}
That the maps $\mu_{w,w'}$ are isomorphisms outside of codimension
two follows from Lemma \ref{keylemma}. For the remaining statements in (1-3) 
consider just the $\shX$ case, as the other cases are identical.
By Lemma \ref{torusgluinglemma2}, each of the $U_{\s,\shX}$ 
is isomorphic, outside of codimension two, to the gluing of the seed torus
$\shX_{\s}$ to its adjacent seed tori $\shX_{\mu_k(\s)}$, $k\in I_{\uf}$.
This gives a birational map $U_{\s,\shX}\dasharrow \shX^{\ft}\subseteq
\shX$. (Here we use any choice of regular subtree of 
$\foT$ containing the vertex
corresponding to $\s$ and its adjacent vertices. The subtree is taken to
be finite but as large as we would like.) Since $\shX^{\ft}$ is
covered, up to codimension two subsets, by some finite collection
$\{U_{\s_w,\shX}\}$ we see each
$U_{\s_w,\shX}$ is isomorphic to $\shX^{\ft}$ off a codimension two
subset. We need to use $\shX^{\ft}$ rather than $\shX$, for if $\shX$
is not Noetherian, the subset of $\shX$ we fail to cover need not be
closed.

\end{proof}

\begin{remark} More generally than the principal coefficient
case, the totally coprime hypothesis 
also holds if the matrix $(\epsilon_{ij})_{i\in I_{\uf}, 1\le j\le n}$ has
full rank. See \cite{BFZ05}, Proposition 1.8. Of course, this holds in
particular for the principal coefficient case.
\end{remark}

We immediately obtain from this a geometric
explanation for the well-known Laurent phenomenon:

\begin{corollary} [The Laurent phenomenon]
\label{lpcor}
For a seed $\s$, let $q \in M^{\circ}$ (resp.\ $q \in N$) have non-negative pairing with
each $e_i$ (resp.\ each $-v_i$) for $i\in I_{\uf}$. 
Equivalently, $z^q$ is a monomial which is a regular function
on the toric variety $\TV_{\s,\shA}$ (resp.\ $\TV_{\s,\shX}$). Then $z^q$ 
is a Laurent polynomial on every seed torus, i.e.,
$z^q\in H^0(\shA,\shO_{\shA})$ (resp. $z^q \in H^0(\shX,\shO_{\shX})$).
\end{corollary}

\begin{proof}
By assumption $z^q$ is a regular function on $\TV_{\s,\shA}$ (or
$\TV_{\s,\shX}$), and hence pulls back and restricts to a regular function
on $U_{\s,\shA}$ (resp.\ $U_{\s,\shX}$). In the $\cX$ case, 
the result then follows from Theorem \ref{maingeometrictheorem}, since
then $z^q$ also defines a regular function on $\shX^{\ft}$ for any choice
of subtree of $\foT$, and hence also defines a regular function on $\shX$.

The $\shA$ case then follows from the $\shA_{\prin}$ case, since the
mutation formula \eqref{Amutationeq} 
for $\shA$ is obtained from that for $\shA_{\prin}$ by setting
$z^{(0,e_i)}=1$ for $(e_1,\ldots,e_n)$ the initial seed.
\end{proof}

\begin{remark} Note that in the $\cA$ case, with no frozen variables, (i.e.,
$I_{\uf}=I$) the condition on $q$ is exactly that
$q$ is in the non-negative span of the $e_i^*$, i.e., 
that $z^q$ is a monomial, with
non-negative exponents, in the cluster variables of the seed. 
In particular, this
applies to any cluster variable, in which case the statement gives the usual 
Laurent phenomenon.
From this point of view the difference between $\cA$ and $\cX$ is that the 
fan $\Sigma_{\s,\cA}$
always looks the same (it is the union of coordinate rays), and in particular 
$\TV_{\s, \cA}$ has lots of global functions (this is a toric open subset of $\bA^n$), while
$\Sigma_{\s,\cX}$ can be any arbitrary collection of rays, and 
$\TV_{\s, \cX}$ has non-constant global functions if and only if
all these rays lie in a common half space. 
\end{remark}

\begin{remark} \label{ourscrewup} 
By \cite{BFZ05}, Def.\ 1.1, the algebra $H^0(U'_{\s,\cA},\cO_{U'_{\s,\cA}})
= H^0(U_{\s,\cA},\cO_{U_{\s,\cA}})$ (see part (4) of Lemma \ref{keylemma})
is the {\it upper bound}. In an earlier version of this paper we claimed Theorem
\ref{maingeometrictheorem} for $\cA$ (without any coprimality assumption), which would in particular imply
the upper bound is equal to the upper cluster algebra. But Greg Muller set
us straight, by giving us an example where the upper cluster algebra is strictly smaller than the upper bound.
\end{remark}

We learned the following theorem, and its proof, from M.\ Shapiro:

\begin{theorem}
\label{separatedtheorem}
The canonical map $\iota:\shA\rightarrow \Spec A^{\up}$ is an open immersion,
where $A^{\up}
=\Gamma(\shA,\shO_{\shA})$ is the upper cluster algebra. In particular,
$\shA$ is separated.
\end{theorem}

\begin{proof}
$\shA$ is covered by open sets of the form $\shA_{\s}$, for various
seeds $\s$. First note that the induced map $\iota_\s:\shA_{\s}\rightarrow\Spec A^{\up}$
is an open immersion. Indeed, this map is induced by the inclusion
$\iota_\s^*:A^{\up}\subseteq \kk[A_1^{\pm 1},\ldots,A_n^{\pm 1}]=:B$, where $A_1,\ldots,
A_n$ are the cluster coordinates on $\shA_\s$. One checks this is
a local isomorphism: given $(a_1,\ldots,a_n)\in \shA_{\s}$, $a_1,\ldots,a_n
\not=0$, the corresponding
maximal ideal is $\fom=\langle A_1-a_1,\ldots,A_n-a_n\rangle\subseteq
B$. By the Laurent phenomenon, $A_1,\ldots,A_n\in A^{\up}$, and thus
$A_1,\ldots,A_n$ are invertible in the localization 
$A^{\up}_{(\iota_\s^*)^{-1}(\fom)}$.
Thus $A^{\up}_{(\iota_\s^*)^{-1}(\fom)}\cong B_{\fom}$, and $\iota_\s$ is a local
isomorphism. Thus by \cite{Gr60}, I, 8.2.8, $\iota_\s$ is an
open immersion. 

To show $\iota$ itself is now an open immersion, it is sufficient to show
it is one-to-one since it is a local isomorphism. Let $x\in \shA_{\s}$,
$y\in \shA_{\s'}$ be such that $\iota(x)=\iota(y)$. Let $A_1,\ldots,A_n$
be the cluster coordinates on $\shA_{\s}$. Again by the Laurent phenomenon,
there is an inclusion $\kk[A_1,\ldots,A_n]\subseteq A^{\up}$, hence a map
$\psi:\Spec A^{\up} \rightarrow \AA^n$. The composition $\psi\circ \iota_{\s}$
is the obvious inclusion and $(\psi\circ\iota_{\s})^{-1}\circ
(\psi\circ\iota_{\s'})$ agrees, as a rational
map, with $\mu_{w',w}$, where $w', w$ are the vertices of $\foT$ corresponding
to the seeds $\s', \s$. Thus the map $\mu_{w',w}$ is defined at $y$,
since $\psi\circ \iota_{\s'}$ is defined at $y$ and $(\psi\circ\iota_{\s})^{-1}$
is defined at $\psi(\iota_{s'}(y))=\psi(\iota_{\s}(x))$. Furthermore,
$\mu_{w',w}$ is then a local isomorphism at $y$ as it agrees with
$\iota_{\s}^{-1}\circ\iota_{\s'}$ at $y$, and $\iota_{\s}$ and $\iota_{\s'}$
are local isomorphisms at $x$ and $y$ respectively.
So the gluing map 
defining $\shA$ identifies $x$ and $y$, and $\iota$ is injective.
\end{proof}

\section{The $\shA_t$ and $\shA_{\prin}$ cluster varieties as torsors}
\label{univtorsorsection}

Fix in this section fixed data  and a seed $\s$ as usual.
We shall assume that there are no frozen variables, i.e., $I_{\uf}=I$,
$N_{\uf}=N$,
and furthermore that the matrix $\epsilon$ has no zero row (or equivalently
no zero column). Note that if $\epsilon$ does have a zero row the
same is true for all mutations, so this condition is mutation independent.
We then obtain the $\shX$, $\shA$, $\shA_{\prin}$ and $\shA_t$ varieties.

Denote by $X$ the open subset of $\shX$ obtained by gluing together
the seed tori $\shX_{\s}$ and $\shX_{\mu_k(\s)}$, $1\le k\le n$. This still
comes with a map $\lambda:X\rightarrow T_{K^*}$ as in Construction
\ref{variousrelations}, and we write
$X_\phi$ for the fibre
over $\phi\in T_{K^*}$.

We first compute the Picard group of $X$ and $X_\phi$:

\begin{theorem}
\label{XePicard}
For $\phi\in T_{K^*}$,
\[
\Pic(X)\cong \Pic(X_\phi)\cong \coker(p^*:N\rightarrow M^{\circ}).
\]
\end{theorem}

\begin{proof}
We first need to describe precisely how $X$ and
$X_\phi$ are glued together out of tori. Let $U_0=\shX_{\s}$,
$U_i=\shX_{\mu_i(\s)}$, $1\le i\le n$.
We have birational gluing maps $\varphi_{ij}:U_i
\dasharrow U_j$ given by $\varphi_{0j}=\mu_j$,
$\varphi_{ij}=\mu_j\circ \mu_i^{-1}$. These glue over sets $U_{ij}$ as in
Proposition \ref{gluingconstruction}. Note that
\[
U_{0j}=\shX_\s\setminus V(1+z^{e_j}),
\]
by \eqref{Xmutationeq} and the fact that no row or column of $\epsilon$
is zero. The same description applies to $U_{j0}$.
On the other hand, noting that
\[
(\mu_j\circ\mu_i^{-1})^*(z^n)=
z^n(1+z^{e_j}(1+z^{e_i})^{[e_j,e_i]})^{[-n,e_j]}(1+z^{e_i})^{[n,e_i]},
\]
one sees that
if we set
\[
h_{ij}=\begin{cases} 1+z^{e_j}(1+z^{e_i})^{[e_j,e_i]}& [e_j,e_i]\ge 0\\
(1+z^{e_i})^{-[e_j,e_i]}+z^{e_j}& [e_j,e_i]\le 0
\end{cases}
\]
then
\[
U_{ij}=\shX_{\mu_i(\s)}\setminus (V(1+z^{e_i})\cup V(h_{ij})).
\]
Now the $U_{ij}$ also map to $T_{K^*}$, with fibres $U_{ij,\phi}$ over $\phi$,
so that $X_\phi$ is obtained by gluing the sets $U_{i,\phi}$ (the fibres of
$U_i\rightarrow T_{K^*}$ over $\phi$) via the restriction
of the $\varphi_{ij}$ to $U_{ij,\phi}$.
Choose a splitting $N=K\oplus N'$. A regular function on a fibre of
$\shX_s\rightarrow T_{K^*}$ is a linear combination of restrictions
of monomials $z^{n'}$, $n'\in N'$, to the fibre.

In particular, we have
\begin{align}
\label{invertibles}
\begin{split}
\Gamma(U_{0i},\shO_{U_{0i}}^{\times})= {} & 
\{ cz^n (1+z^{e_i})^{-a}\,|\, c\in \kk^{\times}, n\in N, a\in \ZZ\},\\
\Gamma(U_{ij},\shO_{U_{ij}}^{\times})= {} & 
\{ c z^n (1+z^{e_i})^{-a} h_{ij}^{-b}\,|\, c\in \kk^{\times}, n\in N, a,b\in\ZZ\},\\
\Gamma(U_{0i,\phi},\shO_{U_{0i,\phi}}^{\times})= {} & 
\{ cz^n (1+z^{e_i})^{-a}\,|\, c\in \kk^{\times}, n\in N', a\in \ZZ\},\\
\Gamma(U_{ij,\phi},\shO_{U_{ij,\phi}}^{\times})= {} & 
\{ c z^n (1+z^{e_i})^{-a} h_{ij}^{-b}\,|\, c\in \kk^{\times}, n\in N', a,b\in\ZZ\},
\end{split}
\end{align}
noting that as $e_i\not\in K$ for any $i$ by assumption on $\epsilon$,
$1+z^{e_i}$ has some zeroes on $U_{ij,\phi}$.
We will now compute $\Pic(X_\phi)$, the argument for $\Pic(X)$ being identical
except that $N'$ is replaced by $N$ below.
We compute $\Pic(X_\phi)=H^1(X_\phi,\shO_{X_\phi}^{\times})$ using
the \v Cech cover $\{U_{i,\phi}\,|\, 0\le i\le n\}$
with $U_{i,\phi}\cap U_{j,\phi}$ identified with $U_{ij,\phi}$ for $i<j$.
Indeed, this cover calculates $\Pic(X_{\phi})$ because $\Pic(U_{i,\phi})=
\Pic(U_{ij,\phi})=0$ for all $i$ and $j$.
Thus a \v Cech 1-cochain consists of elements $g_{ij}\in \Gamma(U_{ij,\phi},
\shO^{\times}_{U_{ij,\phi}})$ for each $i<j$. In particular, if $(g_{ij})$ is
a 1-cocycle, necessarily $g_{ij}=(\mu_i^{-1})^*(g_{0i}^{-1}g_{0j})$,
and the $g_{0i}$'s can then be chosen independently.
From \eqref{invertibles}, the group of 1-cocycles is then identified with
\[
Z^1:=\bigoplus_{i=1}^n(\kk^{\times}\oplus N'\oplus\ZZ).
\]
On the other hand, $\Gamma(U_{i,\phi},\shO_{U_{i,\phi}}^{\times})=
\kk^{\times}\oplus N'$. A
$0$-cochain $g=(g_i)_{0\le i\le n}$, $g_i\in \kk^{\times}\oplus N'$
then satisfies
\[
(\partial g)_{0i}= g_0^{-1} \mu_i^*(g_i),
\]
where $\partial$ denotes \v Cech coboundary.
Given Equation \eqref{Xmutationeq}, we can then view $\partial$ as
a map
\[
C^0:=\bigoplus_{i=0}^n (\kk^{\times}\oplus N')  \rightarrow Z^1,\quad
(c_i,n_i)_{0\le i\le n}  \mapsto (c_ic_0^{-1}, n_i-n_0, [n_i,e_i])_{1\le i\le
n}.
\]
Thus modulo $\partial(C^0)$, every element of $Z^1$ is equivalent to some
$(1,0,a_i)_{1\le i\le n}$. Thus $Z^1/\partial(C^0)$ is isomorphic to
$\ZZ^n/(\partial(C^0)\cap \ZZ^n)$, where $\ZZ^n\subset Z^1$ via the last
component for each $i$. But $\partial(C^0)\cap \ZZ^n$ consists of
the coboundaries of elements $(1,n_0)_{0\le i\le n}$, and the coboundary
of such an element is $(1,0,[n_0,e_i])_{1\le i\le n}$. If we identify
$\ZZ^n$ with $M^{\circ}$ using the basis $f_i$, then with $n_0=e_j$,
we obtain the element of $M^{\circ}$ given by
\[
\sum_{i=1}^n [e_j,e_i]f_i=\sum_{i=1}^n \{e_j,e_i\}e_i^*=p^*(e_j).
\]
This proves the result.
\end{proof}

\begin{remark}
We note the calculations in the above proof demonstrate easily how 
$X_e$ (and hence $X$ and $\shX$)
can fail to be separated. Indeed, suppose that $e_i, e_j$ agree
after projection to $N/K$. In particular, $[n,e_i]=[n,e_j]$ for any
$n\in N$ and $[e_j,e_i]=0$. Thus $\mu_j\circ\mu_i^{-1}$ is the
identity on $\kk[N/K]\cong \shX_{\mu_i(\s),e}\cong \shX_{\mu_j(\s),e}$,
but $U_{ij,e}$ is a proper subset of $\shX_{\mu_i(\s),e}$. So
the two tori are glued via the identity across a proper open subset
of each torus, and we obtain a non-separated scheme.
\end{remark}

\begin{construction} 
\label{universaltorsordef}
We now recall the construction of the universal torsor over a scheme $X$
with finitely generated Picard group.
Ideally, we would like to define the universal torsor as the
scheme affine over $X$
\[
UT_X:= {\bf Spec} \bigoplus_{\shL\in \Pic X} \shL.
\]
However, the quasi-coherent sheaf of $\shO_X$-modules appearing here
doesn't have a natural algebra structure, since elements of
$\Pic X$ represent \emph{isomorphism classes} of line bundles. If
$\Pic X$ is in fact a free abelian group, we can proceed as in
\cite{HK00} and choose a set of
line bundles $\shL_1,\ldots,\shL_n$ whose isomorphism classes form a basis for
the Picard group, and write, for $\nu\in \ZZ^n$, $\shL^{\nu}:=
\bigotimes_{i=1}^n \shL_i^{\nu_i}$. Then 
$\bigoplus_{\nu\in\ZZ^n} \shL^{\nu}$ does have a natural algebra structure.

If $\Pic X$ has torsion, then we need to make use of the definition
given in \cite{BH03}, \S3.
We can choose a sufficiently fine open cover
$\foU$ of $X$ such that every isomorphism class of line bundle on $X$
is represented by a \v Cech 1-cocycle for $\shO^{\times}_X$ 
with respect to this cover. Denoting the set of \v Cech 1-cocycles
as $Z^1(\foU,\shO_X^{\times})$, we choose a finitely generated subgroup
$\Lambda\subseteq Z^1(\foU,\shO_X^{\times})$. If for $\lambda\in\Lambda$
we denote by $\shL_{\lambda}$ the corresponding line bundle, we can
choose $\Lambda$ so that the natural map
$\Lambda\rightarrow\Pic X$, $\lambda\mapsto [\shL_{\lambda}]$, 
is surjective. Then multiplication gives a sheaf of $\shO_X$-algebras
structure to $\shR:=\bigoplus_{\lambda\in\Lambda} \shL_{\lambda}$. 

To obtain the universal torsor, we need to define an ideal $\shI\subseteq\shR$
generated by relations coming from isomorphisms $\shL_{\lambda}\cong 
\shL_{\lambda'}$. However, these isomorphisms must be chosen carefully, 
so \cite{BH03}
defines the notion of a \emph{shifting family}. Let $\Lambda_0=\ker(\Lambda
\rightarrow\Pic X)$. A shifting family is a set of $\shO_X$-module
isomorphisms $\{\rho_{\lambda}:\shR\rightarrow\shR\,|\,\lambda\in\Lambda_0\}$
such that
\begin{enumerate}
\item $\rho_{\lambda}$ maps $\shL_{\lambda'}$ to $\shL_{\lambda'+\lambda}$,
for every $\lambda\in\Lambda_0$, $\lambda'\in \Lambda$;
\item For every $\lambda_1,\lambda_2\in\Lambda_0$, $\rho_{\lambda_1+\lambda_2}
=\rho_{\lambda_1}\circ\rho_{\lambda_2}$;
\item If $f$, $g$ are sections of $\shL_{\lambda_1}$, $\shL_{\lambda_2}$
respectively, for $\lambda_1,\lambda_2\in\Lambda$, and $\lambda\in\Lambda_0$,
we have $\rho_{\lambda}(fg)=f\rho_{\lambda}(g)$.
\end{enumerate}
A shifting family defines a sheaf of ideals $\shI\subseteq \shR$
such that $\shI(U)$ is generated by elements of the form
$f-\rho_{\lambda}(f)$ for $f\in \shR(U)$, $\lambda\in\Lambda_0$.

Given a shifting family, the universal torsor is then defined to be
\[
UT_X:={\bf Spec}\, \shR/\shI.
\]
A priori $UT_X$ depends on the choice of shifting family (although 
\cite{BH03} proves any two choices are isomorphic provided that
$\kk$ is algebraically closed and $\Gamma(X,\shO_X^{\times})=\kk^{\times}$,
see Lemma 3.7 of \cite{BH03}). If $\Pic X$ is torsion free, then this ambiguity
disappears. Thus, in general, we will talk about a \emph{choice of
universal torsor}.
\end{construction}

Given a seed $\s$, $t\in T_M$, let $A_t$ (resp.\ $A_{\prin}$) 
be the variety defined by gluing together
the seed tori $\shA_{\prin,\s,t}$, $\shA_{\prin,\mu_i(\s),t}$ (resp.\ 
$\shA_{\prin,\s}$, $\shA_{\prin,\mu_i(\s)}$), $1\le i\le n$,
analogously
to $X$. 

\begin{theorem}
\label{universaltorsortheorem}
\begin{enumerate}
\item
Let $t\in T_M$, $\phi=i(t)\in T_{K^*}$ (see \eqref{principaltoXdiagram}).
The torsor $p_t:A_t\rightarrow X_\phi$ of Construction \ref{variousrelations}
is a universal torsor for $X_{\phi}$. For very general $t$, $p_t:\shA^{\ft}_t
\rightarrow \shX^{\ft}_{\phi}$ is a universal torsor for $\shX_{\phi}$.
\item
For $m\in M^{\circ}$, let $\shL_m$ denote the line bundle on $X$ associated
to $m$ under the identification $\Pic(X)\cong M^{\circ}
/p^*(N)$. Specifically, $\shL_m$ is the representative of the isomorphism
class given by the \v Cech 1-cocycle represented by $m\in M^{\circ}$ in
the proof of Theorem \ref{XePicard}. Then
\[
A_{\prin}={\bf Spec} \bigoplus_{m \in M^{\circ}} \shL_m.
\]
Furthermore, the line bundle $\shL_m$ on $X$ extends to a line bundle
$\shL_m$ on $\shX^{\ft}$, and similarly
\[
\shA^{\ft}_{\prin}={\bf Spec} \bigoplus_{m \in M^{\circ}} \shL_m
\]
(using the same finite subtrees of $\foT$ to define both $\shA_{\prin}^{\ft}$
and $\shX^{\ft}$).
\end{enumerate}
\end{theorem}

\begin{proof} We first prove the statements for $A_{\prin}$, $X$
and $A_t$, $X_{\phi}$.
Continuing with the notation of the proof of Theorem \ref{XePicard} and
Construction \ref{universaltorsordef}, we take the open covers
$\foU=\{\shX_{\s}\}\cup \{\shX_{\mu_i(\s)}\,|\,1\le i\le n\}=\{U_{i}\,|\,
0\le i\le n\}$,
$\foU_\phi=\{\shX_{\s,\phi}\}\cup \{\shX_{\mu_i(\s),\phi}\,|\,1\le i\le n\}=\{U_{i,\phi}\,|\,
0\le i\le n\}$ 
as usual.
We saw in the proof of Theorem \ref{XePicard} that $M^\circ$ is naturally
identified with a subgroup of both $Z^1(\foU,\shO_X^{\times})$ and
$Z^1(\foU_\phi,\shO_{X_\phi}^{\times})$. Taking
the subgroup $\Lambda$ of this cocycle group to be $M^{\circ}$, we obtain 
\[
\Lambda_0=\ker (\Lambda\rightarrow \Pic(X_\phi))=\ker
(M^{\circ}\rightarrow M^{\circ}/p^*(N))=p^*(N).
\]
This then gives rise to a sheaf of $\shO_{X}$-algebras $\shR=
\bigoplus_{\lambda\in\Lambda}\shL_{\lambda}$ and a sheaf of 
$\shO_{X_\phi}$-algebras $\shR_\phi$ defined by the same formula.
For the two cases, we have the maps $\tilde p:A_{\prin}\rightarrow X$
and $p_t:A_t\rightarrow X_{\phi}$ of Construction \ref{variousrelations}.
Noting that
\[
U_{i,A_{\prin}}:=\tilde p^{-1}(U_{i})=\begin{cases} \shA_{\prin,\s} & i=0\\ 
\shA_{\prin,\mu_i(\s)} & i>0,
\end{cases}
\]
\[
U_{i,A_t}:=p_t^{-1}(U_{i,\phi})=\begin{cases} \shA_{\prin,\s,t} & i=0\\ 
\shA_{\prin,\mu_i(\s),t} & i>0
\end{cases}
\]
we see the morphisms $\tilde p,p_t$ are affine. Thus to prove both parts of the
theorem, it is sufficient to construct
morphisms of sheaves of $\shO_X$-algebras or $\shO_{X_\phi}$-algebras
\begin{equation}
\label{psidef}
\psi:\shR\rightarrow \tilde p_*\shO_{A_{\prin}}, \quad
\psi_\phi:\shR_\phi\rightarrow p_*\shO_{A_t}
\end{equation}
such that $\psi$ is an isomorphism and the kernel of $\psi_\phi$ is an
ideal $\shI$ arising from a shifting family.

First, by construction, 
$\shR_\phi|_{U_{i,\phi}}\cong \bigoplus_{m\in M^{\circ}}
\shO_{U_{i,\phi}} e_m$, and the transition function on $U_{0i,\phi}$ for
the generator $e_m$ is $(1+z^{e_i})^{-\langle d_ie_i,
m\rangle}$. The same formulae hold for $\shR$. 
Let $I_{\phi}\subseteq \Gamma(U_i,\shO_{U_i})=\kk[N]$ be the ideal of
the fibre $U_{i,\phi}\subseteq U_i$, and let $I_t\subseteq 
\Gamma(U_{i,A_{\prin}},\shO_{U_{i,A_{\prin}}})=\kk[\widetilde M^{\circ}]$
be the ideal of the fibre $U_{i,A_t}$ 
of $\pi:U_{i,A_{\prin}}\rightarrow T_M$ over
$t$. Then 
$\shR_\phi|_{U_{i,\phi}}$ is the quasi-coherent sheaf associated
to the free $\kk[N]/I_{\phi}$-module with basis 
$\{e_m\,|\,m\in M^{\circ}\}$, while $\shR|_{U_i}$ is the quasi-coherent
sheaf associated to the free $\kk[N]$-module with the same basis.

Second, note that $\tilde p_*\shO_{U_{i,A_{\prin}}}$ (resp.\ 
$p_*\shO_{U_{i,A_t}}$) is the quasi-coherent sheaf
associated to the $\kk[N]$-algebra $\kk[\widetilde M^{\circ}]$
(resp.\ the $\kk[N]/I_\phi$-algebra $\kk[\widetilde M^{\circ}]/I_t$).
The algebra structure is given by the map $N\rightarrow \widetilde M^{\circ}$,
$n\mapsto (p^*(n),n)$. There are natural
maps $\shR|_{U_i}\rightarrow \tilde p_*\shO_{U_{i,A_{\prin}}}$ and
$\shR_\phi|_{U_{i,\phi}}\rightarrow (p_t)_*\shO_{U_{i,A_t}}$ induced by the maps 
of $\kk[N]$- or $\kk[N]/I_\phi$-modules 
given by $e_m\mapsto z^{(m,0)}$. We first check that
these maps respect the transition maps. 
We do this for the case of $A_t\rightarrow
X_\phi$, the case of $A_{\prin}\rightarrow X$ being identical.
On $U_{0i,\phi}$, $e_m$ is
glued to $(1+z^{e_i})^{-\langle d_ie_i,m\rangle}e_m$ as observed above,
while $z^{(m,0)}\in \kk[\widetilde M^{\circ}]/I_t$ 
is transformed via the $\shA$ mutation
$\mu_i$. But using \eqref{Amutationeq},
\[
{\mu_i^*(z^{(m,0)})\over z^{(m,0)}} = (1+z^{(v_i,e_i)})^{-\langle d_ie_i,m\rangle}.
\]
This can be viewed via $p_t^*$ as the function on $U_{0i,\phi}$ 
given by $(1+z^{e_i})^{-\langle d_ie_i, m
\rangle}$. This shows that the transition maps match up, and we obtain
the desired map \eqref{psidef}.

Note that $\psi$ is easily seen to be an isomorphism. On the other hand,
the kernel $\shI$ of $\psi_\phi$ is generated on $U_{i,\phi}$
by elements of $\bigoplus e_m\kk[N]/I_\phi$ of the form 
$e_m- e_{m+p^*(n)}z^{-n}$
for $m\in M^{\circ}$, $n\in N$. This arises from the family of identifications
$\{\rho_{p^*(n)}\}$
defined by $\rho_{p^*(n)}(e_m)=e_{m+p^*(n)}z^{-n}$.
This is easily checked to be a shifting family. 

This completes the proof for $A_{\prin}$, $X$ and $A_t$, $X_{\phi}$. To
prove the result for $\shA^{\ft}_{\prin}$, $\shX^{\ft}$ etc., one just
notes that the corresponding spaces are equal to $A_{\prin}$, $X$ etc.\
outside of codimension two.
\end{proof}

\begin{definition} Given a choice of shifting family for a scheme $X$ 
over a field $\kk$
with finitely generated Picard group, the \emph{Cox ring} $\Cox(X)$ of $X$ is the 
$\kk$-algebra 
of global sections of $\shR/\shI$. If $\Pic X$ is free, this
coincides with the usual definition
\[
\Cox(X)=\bigoplus_{\nu} \Gamma(X,\shL^{\nu}),
\]
after a choice of line bundles $\shL_1,\ldots,\shL_n$ whose isomorphism
classes give a basis of $\Pic X$.
\end{definition}

\begin{corollary} \label{coxcor}
\begin{enumerate}
\item
The upper cluster algebra with principal coefficients is isomorphic
to
\[
\bigoplus_{m\in M^{\circ}} \Gamma(X,\shL_m).
\]
\item If the initial seed is totally coprime,
the upper cluster algebra is isomorphic to the Cox ring of $X_e$.
\item For $t \in T_M$ very general, $\Gamma(\shA_t,\shO_{\shA_t})$ is
isomorphic to the Cox ring of $X_{i(t)}$.
\end{enumerate}
\end{corollary}

\begin{proof} This follows because with the hypotheses, the upper cluster
algebra $\Gamma(\shA,\shO_{\shA})$ coincides with $\Gamma(A,\shO_A)$ 
by Theorem \ref{maingeometrictheorem} and Lemma \ref{torusgluinglemma2}.
The latter algebra has the desired description by Theorem 
\ref{universaltorsortheorem}. The principal coefficient case is similar.
\end{proof}

\begin{corollary} \label{aprinut} If $\Pic(X)$ is torsion free (i.e., if $M^{\circ}/p^*(N)$ is
torsion free) then the upper cluster algebra with principal coefficients
$\Gamma(\shA_{\prin},\shO_{\shA_{\prin}})$ and for very general $t$,
the upper cluster algebra with general coefficients 
$\Gamma(\shA_t,\shO_{\shA_t})$ are factorial.
If the initial seed is totally coprime, then the upper cluster 
algebra $\Gamma(\shA,\shO_{\shA})$ is
factorial.
\end{corollary}

\begin{proof}
For the cases other than $\shA_{\prin}$, this follows from
Theorem 1.1 of \cite{Ar08},
(see also \cite{BH03}, Prop.\ 8.4.) 

For the principal case, we note that the map $\tilde p:\shA_{\prin}\rightarrow
\shX$ is a $T_{N^{\circ}}$-torsor, and that if $\Pic(X)$ is torsion-free,
then $\Pic(X)^*\subseteq N^{\circ}$ and
$T_{\Pic(X)^*}$ is a subtorus of $T_{N^{\circ}}$. 
Write $X'=A_{\prin}/T_{\Pic(X)^*}$. Note
that $X'={\bf Spec}\, \bigoplus_{m\in p^*(N)} \shL_m$. But $\shL_m\cong
\shO_X$ as a line bundle for $m\in p^*(N)$, so $X'\rightarrow X$ is
a trivial $T_{(p^*(N))^*}$-torsor. In particular, $\Pic(X')\cong\Pic(X)$ and
$A_{\prin}$ is the universal torsor over $X'$. The above cited results show
the Cox ring of $X'$ is a UFD, so the upper cluster algebra with principal
coefficients is also a UFD.
\end{proof}

\section{The $\cX$ variety in the $\rank\epsilon=2$ case} \label{rk2sec}

In this section we will fix seed data as usual, with the same assumptions
as in the previous section, namely that there are no frozen variables
and that no row (or column) of $\epsilon$ is zero. We will assume furthermore
that $\rank \epsilon=2$, i.e., $\rank K = \rank N-2$. In this case, the morphism
$\cX\rightarrow T_{K^*}$ is a flat family of two-dimensional schemes 
(flatness following from the fact that the maps $\shX_\s\rightarrow T_{K^*}$
are flat for each seed). We can use the description of the $\cX$ variety
given in \S\ref{codimtwosection} to develop a geometric feeling for this family.

Now $K^{\perp}\subseteq M$ is a saturated rank two sublattice by the 
assumption on the rank of $\epsilon$. Furthermore, $d_iv_i\in K^{\perp}$
for each $i$ and these vectors are non-zero by the assumption on $\epsilon$.
Choose a complete non-singular 
fan $\oSigma$ in $K^{\perp}$ such that each $-d_iv_i$ generates a 
ray of $\oSigma$.
Via the inclusion $K^{\perp}\subseteq M$, $\oSigma$ also
determines a fan in $M$, which we denote by $\Sigma$. Note $\Sigma$
contains the fan of one-dimensional cones $\Sigma_{\s,\cX}$. Then the 
projection $M\rightarrow K^*\cong M/K^{\perp}$ induces a
map 
\[
\bar\lambda:\TV(\Sigma)\rightarrow T_{K^*},
\] 
each of whose fibres is a complete toric
surface $\TV(\oSigma)$; we in fact have non-canonically 
$\TV(\Sigma)\cong \TV(\oSigma)\times T_{K^*}$, arising from a choice
of splitting $M=K^{\perp} \oplus K^*$. 

Let $D_i$ denote the divisor of $\TV(\Sigma)$ corresponding to the ray
generated by $-d_iv_i$. For each $i$ we obtain a (possibly non-reduced)
hypersurface $Z_i\subseteq D_i$ given by 
\[
Z_i:=D_i\cap\oV\big((1+z^{e_i})^{\ind d_iv_i}\big)
\]
as in \eqref{ZAXdef}. 

\begin{lemma}
\label{sectionlemma}
The underlying closed subset of $Z_i$ is the image of
a section $q_i:T_{K^*}\rightarrow \TV(\Sigma)$ of $\bar\lambda$ if and only if 
the image of $e_i$ in $N/K$ is primitive.
\end{lemma}

\begin{proof}
A choice of splitting $M=K^{\perp}\oplus K^*$ gives a dual splitting
$N\cong N/K\oplus K$. Write $e_i=(e_i',e_i'')$ under this splitting.
The monomial $z^{e_i}$ is 
non-vanishing on $D_i$ as $\langle e_i,-d_iv_i\rangle=0$.
Then restricting $z^{e_i}$ to $\bar\lambda^{-1}(\phi)\cap D_i$ for some
$\phi\in T_{K^*}$, we obtain a monomial $(z^{e_i''}(\phi))\cdot 
z^{e_i'}\in\kk[(d_iv_i)^{\perp}]$, where $(d_iv_i)^{\perp}$ is a sublattice
of $N/K$. Thus $Z_i \cap \bar\lambda^{-1}(\phi)\cap D_i$ consists of a single
point if and only if $e_i'$ is primitive, i.e., the image of $e_i$ is primitive
in $N/K$.
\end{proof}

The following is an enhanced restatement of Theorem \ref{maingeometrictheorem}. 

\begin{lemma}
\label{maintheoremrank2}
Let $\cY\rightarrow T_{K^*}$ be the flat family of
surfaces obtained by blowing up the subschemes $Z_i\subseteq \TV(\Sigma)$
in some order. Let $\cD\subseteq\cY$ be the proper transform of the
toric boundary of $\TV(\Sigma)$, $\lambda:\cY\setminus\cD\rightarrow
T_{K^*}$ the induced map. Then 
\begin{enumerate} 
\item $\cX^{\ft}$ isomorphic to $\cY\setminus\cD$ off of a codimension $\ge 2$
set.
\item If $\phi\in T_{K^*}$ is very general (i.e., in the complement
of a countable union of proper closed subsets),
then $\lambda^{-1}(\phi)$ is isomorphic to
the fibre $\shX^{\ft}_\phi$ of $\shX^{\ft}\rightarrow T_{K^*}$ 
away from codimension two.
\end{enumerate}
\end{lemma}

\begin{proof}
(1) is immediate from Theorem \ref{maingeometrictheorem},
observing that
blowing up the $Z_i$ in some order differs only in codimension two
with the blow-up of the subscheme $\bigcup_i Z_i$.
For (2), we first
use the explicit description of $X$ as described at the beginning of
\S\ref{univtorsorsection}. Indeed, $X$ is obtained by gluing together
tori $U_i$, $0\le i \le n$ as described explicitly in 
the proof of Theorem \ref{XePicard}. Denote by $U_{i,\phi}$, $U_{ij,\phi}$ the
fibres of $U_i, U_{ij}\rightarrow T_{K^*}$ over $\phi$. 
If $Z_i\cap Z_j\cap\bar\lambda^{-1}(\phi)
=\emptyset$, it is then easy to see that the maximal open set of the domain
for which the map $\varphi_{ij}|_{U_{i,\phi}}:U_{i,\phi}
\dasharrow U_{j,\phi}$ is an isomorphism is precisely $U_{ij,\phi}$. 
Thus $X_\phi$ is constructed as the space $X$ is in Lemma \ref{torusgluinglemma2}.
The schemes $Z_i$ in that construction coincide with the schemes
$\bar\lambda^{-1}(\phi)\cap Z_i$. Thus, provided $\phi$ does not lie
in $\bar\lambda(Z_i\cap Z_j)$
for any $i,j$, Lemma \ref{torusgluinglemma2}
applies to show that there is an open immersion $X_\phi \rightarrow \lambda^{-1}
(\phi)$ which is an isomorphism off of a codimension two subset of
$\lambda^{-1}(\phi)$. 

To complete the argument, we follow the proofs of Lemma \ref{keylemma}
and Theorem 
\ref{maingeometrictheorem}. If $\s'=\mu_k(\s)$ and $X'$, $\cY'$, $Z_i'$
etc.\ are constructed using the seed $\s'$, then the argument of Lemma
\ref{keylemma} shows that provided $\phi\not\in \lambda(Z_i\cap Z_j),
\lambda'(Z_i'\cap Z_j')$ for any pair $i\not=j$, $\lambda^{-1}(\phi)$
is isomorphic to $(\lambda')^{-1}(\phi)$ off codimension two. 
Thus $X_{\phi}$ and $X'_{\phi}$ are
isomorphic off a set of codimension two, and the argument is finished
as in Theorem \ref{maingeometrictheorem}.
\end{proof}

Thus the family $\cX^{\ft}\rightarrow T_{K^*}$ can be thought of, away from
codimension two, as a family of surfaces obtained by blowing up a collection
of points on the boundary of a toric variety, and then deleting the 
proper transform of the boundary. In general, since these points are being
blown up with multiplicity, $\cY\setminus\cD$ can be singular. We will first
see that any surface obtained via blowups on the boundary of a toric
surface is deformation equivalent
to any surface in the family $\cY\rightarrow T_{K^*}$ constructed using some 
seed.

\begin{construction}
\label{toricblowupconstruction}
Let $\oY$ be a complete non-singular
toric surface, with toric boundary $\oD$, given by a fan $\oSigma$ in a 
lattice $\oN\cong\ZZ^2$. Choose a collection of irreducible boundary divisors
$\oD_1,\ldots,\oD_n$ (possibly with repetitions) and let $w_i\in\oN$ be the
primitive generator of the ray corresponding to $\oD_i$. Fix positive integers
$\nu_i$, $1\le i\le n$. Suppose that $w_1,\ldots,w_n$ generate $\oN$.

We will use this data to construct seed data, as follows.
Set $N=\ZZ^n$ with basis $\{e_i\}$, $M$ the dual lattice as usual. 
Define a map $\psi:N\rightarrow\oN$ by 
$e_i\mapsto w_i$. By assumption, $\psi$ is surjective.
Choose an isomorphism $\bigwedge^2\oN\cong\ZZ$.
The map $\varphi:\oN\rightarrow M$ given
by $\bar n\mapsto (n\mapsto \psi(n)\wedge \bar n)$ gives a primitive embedding
of the lattice $\oN$ into $M$ by surjectivity of $\psi$.
Let $\nu=\gcd(\nu_1,\ldots,\nu_n)$. 
We then obtain an integer-valued skew-symmetric form $\{\cdot,\cdot\}$ on 
$N$ by 
\[
\{n_1,n_2\}=\nu\psi(n_1)\wedge\psi(n_2)\in\ZZ.
\] 
Note that
$\ker\psi$ coincides with $K=\{ n\in N\,|\, \{n,\cdot\}=0\}$.
Set $d_i=\nu_i/\nu$. This gives us seed data $\{e_i\}$ for the fixed
data $N=N_{\uf}$, $\{\cdot,\cdot\}$ and $\{d_i\}$.

We now analyze the family $\cY\rightarrow T_{K^*}$ arising from this
seed data. Using the inclusion $\varphi$, we write $\Sigma$ for the fan
$\oSigma$ as a fan in $M$. We write $D_i$ for the toric
divisor of $\TV(\Sigma)$ corresponding to the ray generated by $w_i$.
We note that with 
\[
v_i=p^*(e_i)=\{e_i,\cdot\}=-\nu\varphi(\psi(e_i)),
\]
we have $-d_iv_i=\nu_i\varphi(w_i)$. As $\psi$ is surjective, $N/K\cong
\oN$, and the image of each $e_i$ in $N/K$ is primitive, being $w_i\in \oN$. 
Thus by Lemma \ref{sectionlemma}, the closed sets $Z_i$ are images of
sections of $D_i\rightarrow T_{K^*}$. It then follows by
Lemma \ref{maintheoremrank2} that the general fibre of 
$\lambda:\cY\rightarrow T_{K^*}$ is obtained by blowing up $\oY$
at a collection of points $p_1,\ldots,p_n$, with $p_i\in \oD_i$ taken
with multiplicity $\nu_i$.
\end{construction}

We now consider the special case that all $\nu_i=1$. First we note:

\begin{proposition}
\label{niceconditionprop}
Giving
\begin{itemize}
\item fixed data with $\rank N=n$, no frozen variables,
$d_i=1$ for all $i$ and such that
$\{\cdot,\cdot\}$ has rank two and the induced non-degenerate skew-symmetric
pairing on $N/K$ is unimodular, and
\item a seed $\s$ for this fixed data such that the image of each $e_i$
in $N/K$ is primitive;
\end{itemize}
is equivalent to giving primitive vectors 
$w_1,\ldots,w_n\in\oN$ where $\oN$ is a rank two lattice and for which
$w_1,\ldots,w_n$ generate $\oN$.
\end{proposition}

\begin{proof}
Construction \ref{toricblowupconstruction} explains how to pass
from the data of the $w_i$'s to the fixed and seed data. Here we take
$\nu_i=1$ for all $i$ in that construction.
Conversely, given fixed and seed data
as in the proposition, we take $\oN=N/K$, $w_i$ the image of each
$e_i$. The only unimodular integral skew-symmetric pairing on $\oN$, up
to sign, is given by $\{n_1,n_2\}=n_1\wedge n_2$, after a choice of
identification ${\bigwedge}^2\oN\cong\ZZ$. Thus after making a suitable
choice of identification, the given pairing $\{\cdot,\cdot\}$ on
$N$ agrees with the one described in Construction \ref{toricblowupconstruction}.
\end{proof}

Continuing with the above notation, with $\nu_1=\cdots=\nu_n=1$,
consider the family
\[
\lambda:\cY\rightarrow T_{K^*}
\]
of blown up toric
surfaces. In this case, a general
fibre $\lambda^{-1}(\phi)$ for $\phi\in T_{K^*}$ is obtained by blowing up
reduced
points on the non-singular part of the toric boundary of $\oY$. A fibre
$(Y,D)=(\cY_\phi,\cD_\phi)$ is what we call a \emph{Looijenga pair} in 
\cite{GHK12}. Given a cyclic ordering of the irreducible components
of $D=D'_1+\cdots+D'_r$ one gets a canonical identification of the
identity component $\Pic^0(D)$ of $\Pic D$ with $\Gm$, 
see \cite{GHK12}, Lemma 2.1. (We note the divisors $D_1,\ldots,D_n$
are a possibly proper subset 
of $D_1',\ldots,D_r'$, and the former occur with repetitions and need
not be cyclically ordered). Furthermore, we
define $D^{\perp}\subseteq \Pic(Y)$ by 
\[
D^{\perp}:=\{ H\in \Pic(Y)\,|\, H\cdot D'_i=0 \quad \forall i\}.
\]
Then the \emph{period point} of $(Y,D)$ is the element of $\Hom(D^{\perp},
\Pic^0(D))$ given by restriction.

\begin{theorem}
\label{Xisuniversalfamily}
Let $\pi:Y\rightarrow\oY$ be the blow-up describing $Y$, with
exceptional divisors $E_1,\ldots,E_n$ over $Z_1\cap \bar\lambda^{-1}(\phi),
\ldots,Z_n\cap\bar\lambda^{-1}(\phi)$. 
Then there is a natural isomorphism
\begin{align*}
K\rightarrow& D^{\perp}\\
\sum a_ie_i\mapsto& \pi^*C-\sum a_iE_i
\end{align*}
where $C$ is the unique divisor class such that 
\begin{equation}
\label{Cdef}
C\cdot D'_j=\sum_{i:D_i=D'_j} a_i.
\end{equation}
Under this identification, and the canonical identification
$\Pic^0(D)\cong\Gm$, the period point of $(Y,D)$ in $\Hom(D^{\perp},\Pic^0(D))$ 
coincides with $\phi\in T_{K^*}=\Hom(K,\Gm)$.
\end{theorem}

\begin{proof}
Recall the standard description of the second homology group of the
toric variety $\oY$:
\[
0\rightarrow H_2(\oY,\ZZ)\rightarrow \ZZ^r\rightarrow \oN\rightarrow 0,
\]
where the map $\ZZ^r\rightarrow\oN$ takes the $i^{th}$ generator of
$\ZZ^r$ to the primitive generator of the ray of $\oSigma$ corresponding
to the divisor $\oD'_i$. The inclusion $H_2(\oY,\ZZ)\hookrightarrow\ZZ^r$
is given by $\alpha\mapsto (\alpha\cdot D_i')_{1\le i\le r}$. 
Since $\oY$ is a non-singular proper rational surface, 
we have $H_2(\oY,\ZZ)\cong\Pic(\oY)$. In particular, if
$\sum a_ie_i\in K$, then $(\sum_{j: D_j=D_i'} a_j)_{1\le i\le r}
\in H_2(\oY,\ZZ)$. Thus there is a unique element $C\in H_2(\oY,\ZZ)
\cong\Pic(\oY)$ satisfying \eqref{Cdef}. It is then clear
that $\pi^*C-\sum a_iE_i\in D^{\perp}$. That this is an isomorphism
is easily checked.

We now need to calculate $\shO_{Y}(\pi^*C-\sum a_iE_i)|_{D}$. As the
identification $\Pic^0(D)$ with $\Gm$ requires a choice of cyclic ordering
of the $D_i'$, or equivalent of the $\oD_i'$, we order $w_1',\ldots,w_r'$
clockwise as defined using the choice of isomorphism $\bigwedge^2
\oN\rightarrow \ZZ$. In particular, this choice of isomorphism also
allows an identification of $\oN$ with $\oM=\Hom(\oN,\ZZ)$, via
$n \in \bar N$ acts by $n' \mapsto (n'\wedge n)$. Thus, in particular,
$z^{w_i'}$ can be viewed as a coordinate on $D_i'$ which is zero 
on $p_{i,i+1}$, the intersection point of $D_i'$ and $D_{i+1}'$, and
infinite at $p_{i-1,i}$ (all indices taken modulo $r$).

We next note that $\shO_{\oY}(C)|_{\oD}$ was calculated in the proof of
\cite{GHK12}, Lemma 2.6, (1). Let $m_i\in D_i'$ be the point where
$z^{w_i'}$ takes the value $-1$. Then
\[
\shO_{\oY}(C)|_{\oD}=\shO_{\oD}(\sum_{j=1}^r(C\cdot \bar D'_j)m_j).
\]
Thus we have the same identity for the restriction of $\shO_Y(\pi^*C)$
to $D$. So if $E_i\cap D=p_i$, we then have
\[
\shL:=\shO_{Y}(\pi^*C-\sum a_iE_i)|_{D}\cong 
\shO_D\big(-\sum_{i=1}^na_ip_i+\sum_{j=1}^r (C\cdot\bar D_j')m_j\big).
\]
This line bundle is described under the isomorphism $\Pic^0(D)\cong\Gm$
of \cite{GHK12}, Lemma 2.6, as follows. 
We have $\shL|_{D_j'}=\shO_{D_j'}((C\cdot \bar D_j')m_j-\sum_{i:D_i=D_j'} a_ip_i)$.
Viewing this trivial 
line bundle as a subsheaf of the sheaf of rational functions, and
using a splitting $M=K^{\perp}\oplus K^*$, $N=N/K\oplus K$ as in the proof
of Lemma \ref{sectionlemma},
a trivializing section is given by the rational function
\[
\sigma_j:={\prod_{i: D_i=D_j'} (z^{w_j'}\cdot z^{e_i''}(\phi)+1)^{a_i}
\over
(z^{w_j'}+1)^{C\cdot \bar D'_j}}
\]
since $Z_i$ is given by the equation $z^{e_i}+1=0$ and under the choice
of splitting $e_i=(w_i,e_i'')$, with $w_i=w_j'$ if $D_i=D_j'$. The image of 
the line bundle in $\Gm$ is
\[
\prod_{j=1}^r \sigma_{j+1}(p_{j,j+1})/\sigma_j(p_{j,j+1})=
\prod_{i=1}^n (z^{e_i''}(\phi))^{a_i}.
\]
Remembering that we are viewing $\phi\in \Hom(K,\Gm)$, we see that
$z^{e_i''}(\phi)=\phi(e_i'')$. Note that if $\sum a_ie_i\in K$, we have
$\sum a_ie_i=\sum a_ie_i''$. Thus we see that the element of $\Gm$
corresponding to our line bundle is precisely $\phi(\sum a_ie_i)$. 
Thus $\phi$ is the period point of $(Y,D)$.
\end{proof}

This shows that the families $\cY\rightarrow T_{K^*}$ agree with the
universal families of Looijenga pairs constructed in \cite{GHK12}.

We can also use the above observations to define an unexpected mutation
invariant in the situation of Proposition \ref{niceconditionprop}.

\begin{theorem}
\label{symmetricform}
Given fixed data and seed data satisfying the conditions of Proposition
\ref{niceconditionprop}, the isomorphism $K\cong D^{\perp}$ of 
Theorem \ref{Xisuniversalfamily} induces a symmetric integral pairing
on $K$ via the intersection pairing on $D^{\perp}\subseteq \Pic(Y)$.
This symmetric pairing on $K$ is independent of mutation.
\end{theorem}

\begin{proof} 
It is enough to check the independence under a single mutation
$\mu_k$. So suppose given seeds $\s, \s'=\mu_k(\s)$. These two seeds
give rise to families $\cY,\cY'\rightarrow T_{K^*}$.
Fix $\phi\in T_{K^*}$ sufficiently general so that
the fibres $\cY_{\phi}$ and $\cY'_{\phi}$ are both blowups of toric
varieties $\oY$, $\oY'$ at \emph{distinct} points. These toric varieties
are related by an elementary transformation as follows. 
Let $\psi:N\rightarrow \oN=N/K$ be the quotient
map. If $\oSigma,\oSigma'$ are the fans in $\oN$ defining
$\oY$ and $\oY'$ respectively, then $\oSigma'$ is obtained from $\oSigma$
as follows. First, we can assume that both $\oSigma$ and $\oSigma'$
contain rays generated by $w_k$ and $-w_k$. Then we can assume
the remaining rays of $\oSigma'$ are obtained by applying the piecewise
linear transformation
\[
T_k:\bar n\mapsto \bar n+[\bar n \wedge w_k]_+ w_k.
\]
Note in particular that this is compatible with \eqref{etransform}. 

The map $\oN\rightarrow \oN/\RR w_k$ defines $\PP^1$-fibrations
$g:\oY\rightarrow\PP^1$, $g':\oY'\rightarrow\PP^1$. 

By Lemma \ref{elemtransformlemma2}, the seed mutation $\mu_k:T_N\dasharrow
T_N$ extends to a birational map $\mu_k:\cY\rightarrow \cY'$
which is an isomorphism off of sets of codimension $\ge 2$. In fact,
one checks easily from the details of the proof that this birational
map restricts to a biregular isomorphism between $Y$ and $Y'$. 
Specifically, this isomorphism is described as follows. Let $p_1,\ldots,p_n$
be the points of $\oY$ blown up to obtain $Y$, and $p_1',\ldots,p_n'$
the points of $\oY'$ blown up to obtain $Y'$. Then if $Y_k, Y_k'$ denote the
blowup of $\oY,\oY'$ at $p_k,p_k'$ respectively, we already have an
isomorphism $\mu_k:Y_k\rightarrow Y_k'$, and $p_i'=\mu_k(p_i)$ for
$i\not=k$. The isomorphism $\mu_k:Y\rightarrow Y'$ is then
obtained by further blowing up the $p_i, p'_i$ for $i\not=k$.
Furthermore, the composition $Y_k\mapright{\mu_k} Y_k'\rightarrow
\oY'$ contracts the proper transform of the curve $F_k=g^{-1}(g(p_k))$ to
$p_k'$. In particular, the curve class $F_k-E_k\in \Pic(Y)$ is mapped
to $E_k'$.

We now need to check that the composition of isomorphisms
$(D')^{\perp}\cong K\cong D^{\perp}$ given in Theorem 
\ref{Xisuniversalfamily} coincides with $\mu_k^*:(D')^{\perp}\rightarrow
D^{\perp}$. To do so, suppose given $\sum a_ie_i=\sum a_i'e_i'\in K$.
Then $a_i=a_i'$ for $i\not=k$ and $a_k'=-a_k+\sum_i [\epsilon_{ik}]_+a_i$
by \eqref{etransform}. These determine classes $C\in\Pic(\oY)$, $C'\in
\Pic(\oY')$ as in Theorem \ref{Xisuniversalfamily}. It is enough
to check that if $\pi_1:Y_k\rightarrow \oY$, $\pi_1':Y_k'\rightarrow\oY'$
are the blowups at $p_k$, $p_k'$ respectively, then 
$\mu_k^*((\pi'_1)^*(C')-a_k'E_k')=\pi_1^*(C)-a_kE_k$. Call these two
divisors $C_1'$ and $C_1$ respectively. Since the Picard
group of $Y_k$ is easily seen to be generated by the proper transforms
of the toric divisors of $Y_k$, and $\mu_k$ takes the boundary
divisor of $Y_k$ corresponding to $\rho\in\oSigma$ to the boundary divisor
of $Y_k'$ corresponding to $T_k(\rho)\in\oSigma'$, it is
enough to check that $C_1'$ and $C_1$ have the same intersection
numbers with the boundary divisors of $Y_k$. It is clear that
$C_1'$ and $C_1$ have the same intersection numbers with all boundary
divisors except possibly those corresponding to the rays $\pm\RR_{\ge 0}w_k$.
Call these two divisors $D_{k,\pm}\subseteq Y_k$. Then
\[
C_1\cdot D_{k,+}=\sum_{i: w_i=w_k, i\not=k} a_i=C_1'\cdot D_{k,+},
\]
while
\[
C_1\cdot D_{k,-}=\sum_{i: w_i=-w_k, i\not=k} a_i=C_1'\cdot D_{k,-}.
\]
This proves the result.
\end{proof}

\section{Examples of non-finitely generated upper cluster algebras} \label{nfgsec}

We will now use the material of the previous two sections to construct
examples of non-finitely generated upper cluster algebras with principal
coefficients and with general
coefficients. These examples are
a generalization of the example of Speyer \cite{Sp13}. They fail to be
finitely generated as a consequence of the following:

\begin{lemma} Let $A$ be a ring, $R$ an $M=\ZZ^n$-graded $A$-algebra,
$R=\bigoplus_{m\in M} R_m$. If $R_0=A$ and $R_m$ is not a finitely generated
$A$-module for some $m\in M$, then $R$ is not Noetherian.
\end{lemma}

\begin{proof} 
Let $I$ be the homogeneous ideal of $R$ generated by $R_m$. We show $I$ is not
finitely generated as an ideal.
Suppose to the contrary that homogeneous $f_1,\ldots,f_p\in R$ generate $I$.
Necessarily $f_i=\sum r_{ij} f_{ij}$ for some $r_{ij}\in R$ and 
$f_{ij}\in R_m$, so we can assume $I$ is generated by a finite number of 
$f_{ij}\in R_m$. But $R_m=I\cap R_m$ is the $A=R_0$-module generated by
the $f_{ij}$, contradicting the assumption that $R_m$ is not finitely
generated as an $A$-module.
\end{proof}

In what follows, suppose given fixed data and seed data satisfying the
hypotheses of Proposition \ref{niceconditionprop}. This gives rise to the
family $\lambda:(\cY,\cD)\rightarrow T_{K^*}$ of Loojienga pairs,
obtained by blowing up a sequence of centers $Z_1,\ldots,Z_n\subseteq
\oY\times T_{K^*}$ in some order, for $\oY$ a toric surface.

\begin{theorem}
\label{nonfgtheorem}
Let $(Y,D)$ be the general fibre of
$(\cY,\cD)\rightarrow T_{K^*}$. Suppose that every 
irreducible component of $D$ has self-intersection $-2$. Then
\begin{enumerate}
\item $\Gamma(\shA_{\prin},\shO_{\shA_{\prin}})$ is non-Noetherian.
\item For $t\in T_M$ very general, $\shA_t$ the corresponding cluster
variety with general coefficients, then $\Gamma(\shA_t,\shO_{\shA_t})$ is
non-Noetherian.
\end{enumerate}
\end{theorem}

\begin{proof}
Let $X$ be as usual the subset of
$\cX$ obtained by gluing together the initial seed torus $\shX_{\s}$ and
adjacent seed tori $\shX_{\mu_k(\s)}$. By Corollary \ref{coxcor},
$\Gamma(\shA_{\prin},\shO_{\shA_{\prin}})=\bigoplus_{m\in M^{\circ}}
\Gamma(X,\shL_m)$, and this gives a $M^{\circ}$-grading on this
algebra. In addition, by Lemma \ref{maintheoremrank2}, $X$ and 
$\cY\setminus\cD$ agree off of a codimension $\ge 2$ set, and both $X$ and
$\cY\setminus\cD$ are non-singular. So
$\Pic X\cong \Pic(\cY\setminus \cD)$.
Thus for each $m\in M^{\circ}$, $\shL_m$ can be viewed as a line bundle
on $\cY\setminus\cD$, and $\shL_m$ has the same space of sections regardless
of whether $\shL_m$ is viewed as a bundle on $X$ or on $\cY\setminus\cD$.
So, by the lemma, it will suffice to show that $\Gamma(\cY\setminus\cD,
\shO_{\cY})=A:=\kk[K]$ and find some line bundle $\shL$ on
$\cY\setminus\cD$ such that $\Gamma(\cY\setminus\cD,\shL)$ is not a
finitely generated $A$-module.

To show $\Gamma(\cY\setminus\cD,\shO_{\cY})=A$, it is sufficient
to show that a regular function on $\cY\setminus\cD$ must be constant
on the very general fibre of $\lambda:\cY\setminus\cD\rightarrow T_{K^*}=
\Spec\kk[K]$.
So consider the fibre $(Y,D)$ of $(\cY,\cD)\rightarrow
T_{K^*}$ over $\phi\in T_{K^*}$. 
The space of regular functions on $Y\setminus D$ can be identified
with $\lim_{\rightarrow} H^0(Y,\shO_Y(nD))$. Consider the long exact
cohomology sequence associated to
\[
0\rightarrow \shO_Y(nD)\rightarrow \shO_Y((n+1)D)\rightarrow 
\shO_Y((n+1)D)|_{D}\rightarrow 0.
\]
Note that $D\in D^{\perp}$ since all components of $D$ have square $-2$
and with $\phi\in T_{K^*}=\Hom(D^{\perp},\Pic^0(D))$, $\phi(D)$ is the normal
bundle of $D$ in $Y$ by Theorem \ref{Xisuniversalfamily}. Thus as
$\phi$ is very general, $\phi(D)$ is not torsion.
So $H^0(D,\shO_D(nD))=0$ for all $n>0$, 
and we see that $H^0(Y,\shO_Y(nD))=\kk$
for all $n\ge 0$. Thus the only regular functions on $Y\setminus D$
are constant.

Now let $E$ be the exceptional divisor over the last center $Z_i$
blown up in constructing $\cY$, so that $E$ is a $\PP^1$-bundle over
$T_{K^*}$. Then we claim $\Gamma(\cY\setminus\cD,\shO_{\cY}(E))$
is not a finitely generated $A$-module. Note that 
\[
\Gamma(\cY\setminus\cD,\shO_{\cY}(E))=\lim_{\longrightarrow\atop
n\ge 0} \Gamma(\cY,\shO_{\cY}(E+n\cD)).
\]
Since each of these groups is an $A$-module, it is sufficient to show 
that the increasing chain of $A$-modules
\begin{equation}
\label{sectionchain}
\Gamma(\cY,\shO_{\cY}(E))\subseteq
\Gamma(\cY,\shO_{\cY}(E+\cD))\subseteq
\Gamma(\cY,\shO_{\cY}(E+2\cD))\subseteq\cdots
\end{equation}
does not stabilize. We have a long exact sequence
\begin{align*}
0&\mapright{} H^0(\cY,\shO_{\cY}(E+(n-1)\cD))
\mapright{i}H^0(\cY,\shO_{\cY}(E+n\cD))
\mapright{} H^0(\cD,\shO_{\cD}(E+n\cD))\\
&\mapright{} H^1(\cY,\shO_{\cY}(E+(n-1)\cD))
\mapright{} H^1(\cY,\shO_{\cY}(E+n\cD))
\mapright{} H^1(\cD,\shO_{\cD}(E+n\cD)).
\end{align*}
If $(Y,D)$ is any fibre of $(\cY,\cD)\rightarrow T_{K^*}$,
one checks easily that $H^1(Y,\shO_Y(E\cap Y))=0$ 
(as $E\cap Y$ is an irreducible $-1$-curve) and that $H^1(D,\shO_D((E\cap D)
+nD))=0$ (using that $E\cap D$ consists of one point). It then
follows from cohomology and base change along with the fact that
$T_{K^*}$ is affine that $H^1(\cD,\shO_{\cD}(E+n\cD))=0$ for all
$n\ge 0$ and $H^1(\cY,\shO_{\cY}(E))=0$. Inductively from the above
long exact sequence one sees $H^1(\cY,\shO_{\cY}(E+n\cD))=0$ for all $n\ge 0$.
Thus the cokernel of the inclusion $i$ in the above long exact sequence is
$H^0(\cD,\shO_{\cD}(E+n\cD))$. Now $\lambda_*\shO_{\cD}(E+n\cD)$ is
a line bundle on $T_{K^*}$, again by cohomology and base change, and
since $T_{K^*}$ is an algebraic torus, this line bundle must be trivial.
Thus $H^0(\cD,\shO_{\cD}(E+n\cD))=A$, and we see the chain \eqref{sectionchain}
never stabilizes. 

The argument for $\shA_t$ for $t$ very general is identical but easier,
as we have already done the relevant cohomology calculations on $(Y,D)$
a very general fibre of $(\cY,\cD)\rightarrow T_{K^*}$.
Then one makes use of Corollary \ref{coxcor}, (3). 
\end{proof}

\begin{example}
Using Construction \ref{toricblowupconstruction} it is easy to produce
many examples satisfying the hypotheses of the above theorem.
For example, let $\oSigma$ be the fan for $\PP^2$, with rays generated
by $w_1=w_2=w_3=(1,0)$, $w_4=w_5=w_6=(0,1)$ and $w_7=w_8=w_9=(-1,-1)$. 
Take all $\nu_i=1$.
Thus a general $(Y,D)$ involves blowing three points on each of the 
coordinate lines of $\PP^2$, so $D$ is a cycle of three $-2$ curves.

This is very closely related to the example of Speyer \cite{Sp13}, 
which in the terminology of Construction \ref{toricblowupconstruction}
again involves taking the fan for $\PP^2$, $w_1=(1,0)$, $w_2=(0,1)$ and
$w_3=(-1,-1)$, but taking 
all $\nu_i=3$. The surface $(Y,D)$ will be constructed 
by performing a weighted blowup of one point on each of three coordinate lines
on $\PP^2$. Then $D$ is still a cycle of three $-2$ curves, but the
situation requires some additional analysis because $Y$ is in fact singular
(having three $A_2$ singularities).
\end{example}

\begin{remark}
In fact there is a much broader range of counterexamples: suppose that
the blowup $(\cY,\cD)\rightarrow \oY\times T_{K^*}$ factors through 
$(\cY,\cD)\rightarrow (\cY',\cD')$, such that a very general fibre $(Y',D')$
of $(\cY',\cD')\rightarrow T_{K^*}$ has the property that every irreducible
component of $D'$ has self-intersection $-2$. Then the argument above
shows that the Cox ring of $Y'\setminus D'$ is non-Noetherian, and
$Y'\setminus D'$ is an open subset of $Y\setminus D$. If $U\subseteq V$,
then the Cox ring of $V$ surjects onto the Cox ring of $U$, so
the fact that the Cox ring of $Y'\setminus D'$ is non-Noetherian
implies the Cox ring of $Y\setminus D$ is non-Noetherian. A similar
but slightly more delicate argument also applies to the principal coefficient
case.

In fact, we expect that whenever the intersection matrix of $D$ is negative
definite, the Noetherian condition fails.
\end{remark}

\section{Counterexamples to the Fock-Goncharov dual bases conjecture}

\label{FGcounterexamplessection}

\cite{FG09} gave an explicit conjecture about the existence of $\kk$-bases for
the $\shX$ and $\shA$ cluster algebras. We will state it loosely here, 
under the assumption that all $d_i=1$, so that $M= M^\circ$, $N= N^\circ$. This
merely allows us to avoid discussing Langlands dual seeds.

Fock and Goncharov conjecture
\begin{conjecture}[\cite{FG09}, 4.1] $N$ parameterizes a canonical basis
of $H^0(\cX,\cO_{\cX})$, and $M$ parameterizes a canonical basis of sections
of $H^0(\cA,\cO_{\cA})$.
\end{conjecture}

In fact, the conjecture as stated in \cite{FG09} is much stronger, giving
an explicit conjectural description of the bases as the set of positive 
universal Laurent polynomials which are extremal, i.e., not a non-trivial
sum of two other positive universal Laurent polynomials. This strongest
form of the conjecture has now been disproven in \cite{LLZ13}, in which
examples are given where the set of all extremal positive universal Laurent
polynomials are not linearly independent. Here we give a much more basic
counterexample to a much weaker form of the conjecture.

We shall again restrict to the case that that there are no frozen
variables. We will merely assume the conjectured basis is compatible with the
$T_K$ action on $\shA$ given by Remark 
\ref{seedtoriremarks}, (3),
and the map $\lambda:\shX\rightarrow T_{K^*}$ in
the natural way. 
We assume that the canonical basis element of $\Gamma(\shX,\shO_{\shX})$
corresponding to $n\in K$ is $\lambda^*(z^n)$. Furthermore, for $\pi:
M\rightarrow K^*$ the natural projection dual to the inclusion $K\rightarrow N$,
we assume that the set $\pi^{-1}(m)$ parameterizes a basis of
the subspace of $H^0(\shA,\shO_{\shA})$ of weight $m$ 
eigenvectors for the $T_K$ action.

We indicate now why a basis with these properties cannot exist in general. 
We consider the rank $2$ cluster algebras produced by Construction 
\ref{toricblowupconstruction}, following
the notation of the construction, taking all $\nu_i=1$. The
general fibre of $\lambda:\cX \to T_{K^*}$ is isomorphic up to codimension
two subsets to the general fibre of $\lambda:\cY\setminus\cD
\rightarrow T_{K^*}$. A fibre of the latter map is of the form
$U:=Y\setminus D$, where $(Y,D)$ is a Looijenga pair
with a map $Y \to \oY$  obtained by blowing up points
on the toric boundary. Since the initial data of the $w_i$ in Construction
\ref{toricblowupconstruction} can be chosen arbitrarily, and in 
particular the $w_i$'s may be repeated as many times as we like,
we can easily find examples for which $D \subseteq Y$ is
analytically contractible, i.e., there is an analytic map 
$(Y,D) \to (Y',p)$ with exceptional locus
$D$ and $p\in Y'$ a single point. Further,
$U = Y \setminus D = Y' \setminus \{p\}$,
and so $H^0(U,\cO_U) = H^0(Y',\cO_{Y'}) = \kk$. 
Even if $D$ is not contractible but rather
a cycle of $-2$ curves as in Theorem \ref{nonfgtheorem}, 
the very general fibre
$Y\setminus D$ will only have constant functions.
It follows that 
\[
H^0(\cX,\cO_{\cX}) = H^0(T_{K^*},\cO_{T_{K^*}}) = \kk[K].
\]
Thus there are no functions for points of $N \setminus K$ to parameterize. 

Consider the conjecture in the opposite direction. Assume for simplicity that,
as in Corollary \ref{aprinut}, $\Pic(X) = \Pic(X_t) = M/p^*(N)$ is
torsion free. Then
$(M/p^*(N))^* = K = \Ker(p^*: N \to M)$. The Fock-Goncharov conjecture
for $\cA_{\prin}$ implies the analogous result for very general $\cA_t$, i.e. the
existence of a canonical
basis of the upper cluster algebra with very general coefficients $H^0(\cA_t,\cO)$,
parameterized by $\cX^{\trop}(\bZ)$. We have 
\[
H^0(\cA_t,\cO_{\cA_t}) = \Cox(X_{i(t)}) = \bigoplus_{m \in K^*= \Pic(X_{i(t)})} H^0(X_{i(t)},\shL_m)
\]
by Corollary \ref{coxcor}. Here $\shL_m$ is a line bundle representing the
isomorphism class given by $m$.
Assuming the canonical bases are compatible with the
natural torus actions, then for $m\in K^*$,  $\pi^{-1}(m) \subset \cX^{\trop}(\bZ)$ restricts to a basis for
the weight $m$-eigenspace $H^0(X_{i(t)},\shL_m)$ of $H^0(\cA_t,\cO)$ under the $T_K$ action, for
$\pi: \cX^{\trop}(\bZ) \to K^* = M/p^*(N)$, the natural map induced by the fibration $\cX \to T_{K^*}$. 
But any choice of seed identifies $\cX^{\trop}(\bZ)$ with $M$ and 
each fibre of $\pi$ with a $p^*(N)=N/K$ torsor. 
Thus the conjecture implies all
line bundles on $X_t$ have isomorphic spaces of sections, 
with basis parameterized (after choice
of seed) by an $N/K$ affine space. This is a 
very strong condition --- most varieties have line bundles with no non-trivial
sections, and rather than an affine space one would expect (for example by comparison with the toric case) 
sections parameterized by integer points of a polytope. Explicitly, the example
of Theorem \ref{nonfgtheorem} clearly has line bundles with non-isomorphic
spaces of sections.

This reasoning suggests to us the conjecture can only hold if $\cX$ is affine {\it up to flops}:

\begin{conjecture} \label{cfgconj} If the Fock-Goncharov conjecture holds 
then $H^0(\cX,\cO_{\cX})$ is finitely
generated, and the canonical map
$\cX \to \Spec(H^0(\cX,\cO_{\cX}))$ is an isomorphism, outside of codimension two. 
\end{conjecture} 

The results of \S \ref{rk2sec} imply that when the hypotheses of Theorem
\ref{symmetricform} are satisfied, the conditions in 
Conjecture \ref{cfgconj} hold iff the generic fibre of $\cX \to T_{K^*}$ is
affine, which is true iff
the canonical symmetric form on $K$ given by Theorem \ref{symmetricform}
is negative definite. Indeed, the generic fibre of $\cX\to T_{K^*}$ is
isomorphic, up to codimension two, to a surface $Y\setminus D$ as
in Theorem \ref{Xisuniversalfamily}. But if $Y\setminus D$ is affine, then
$D$ supports an ample divisor, and by the Hodge index theorem, $D^{\perp}$
is negative definite. Conversely, if $D^{\perp}$ is negative definite,
there must be some integers $a_i$ such that $(\sum a_i D_i)^2>0$.
The result then follows from \cite{GHK11}, Lemma 6.8 and the fact
that $(Y,D)$ is chosen generally in the family.

\end{document}